\documentclass[11pt,english,a4paper]{smfart}
\usepackage[english]{babel}
\usepackage{xcolor}
\usepackage{pbsi}
\usepackage[T1]{fontenc}
\usepackage{stmaryrd}
\usepackage{mathabx}
\usepackage{dsfont}
\usepackage[mathscr]{euscript}
\usepackage{cancel}
\usepackage{array}
\usepackage{float}
\usepackage{placeins}
\usepackage[pagebackref,colorlinks=true,linkcolor=blue,citecolor=blue,urlcolor=red, hypertexnames=true]{hyperref}

 \usepackage[abbrev,backrefs,nobysame]{amsrefs}
 \usepackage{caption}
 \usepackage{enumerate}

\usepackage{amsmath,amssymb,bm,amsfonts}
\usepackage{nicematrix}
\usepackage[all]{xy}
\entrymodifiers={+!!<0pt,\fontdimen22\textfont2>}

\usepackage{mathrsfs}

\usepackage{graphicx}

\usepackage{color}
\usepackage{comment}

\newtheorem{theorem}{Theorem}[section]
\newtheorem{lemma}[theorem]{Lemma}

\newtheorem{corollary}[theorem]{Corollary}
\newtheorem{proposition}[theorem]{Proposition}

\newcounter{hypothesis}
\renewcommand{\thehypothesis}{(H$_{\arabic{hypothesis}}$)}

\theoremstyle{remark}

\newtheorem*{remark*}{\it Remark}

\newcommand{\flba}[2]{
\xymatrix@C15pt{#1\ar@{|->}[r]&#2}}
\newcommand{\flcourte}[2]{
\xymatrix@C12pt{#1\ar[r]&#2}}

{\theoremstyle{definition}}

{\theoremstyle{definition}\newtheorem{example}[theorem]{Example}}

\theoremstyle{definition}

{\theoremstyle{definition}}

\def\T{\ensuremath{\mathbb T}}
\def\R{\ensuremath{\mathbb R}}
\def\Z{\ensuremath{\mathbb Z}}

\def\C{\ensuremath{\mathbb C}}

\def\N{\ensuremath{\mathbb N}}
\def\P{\ensuremath{\mathbb P}}
\def\E{\mathbb E}

\def\bth{\begin{theorem}}
\def\blm{\begin{lemma}}
\def\bpr{\begin{proposition}}
\def\bpf{\begin{proof}}
\def\epf{\end{proof}}
\def\epr{\end{proposition}}
\def\elm{\end{lemma}}
\def\eth{\end{theorem}}
\def\bco{\begin{corollary}}
\def\eco{\end{corollary}}
\def\be{\begin{enumerate}}
\def\ee{\end{enumerate}}
\def\bea{\begin{enumerate}[\rm (a)]}
\def\beun{\begin{enumerate}[\rm (1)]}
\def\bei{\begin{enumerate}[\rm (i)]}

\renewcommand{\Im}{\operatorname{Im}}
\renewcommand{\Re}{\operatorname{Re}}

\newcommand{\vertiii}[1]{{\left\vert\kern-0.25ex\left\vert\kern-0.25ex\left\vert #1 
    \right\vert\kern-0.25ex\right\vert\kern-0.25ex\right\vert}}

\newcommand{\lan}{\langle}
\newcommand{\ran}{\rangle}

\newcommand{\Omg}{\Omega}

\newcommand{\indic}{\mathds{1}}

\numberwithin{equation}{section}

\author[Valentin Gillet]{Valentin Gillet}
\address[Valentin Gillet]{Univ. Lille, CNRS, UMR 8524 - Laboratoire Paul Painlevé, F-59000 Lille, France}
\email{valentin.gillet@univ-lille.fr}

\thanks{This work was supported in part by the project COMOP of the French National Research Agency (grant ANR-24-CE40-0892-01). The author also acknowledges the support of the CDP C2EMPI, as well as the French State under the France-2030 programme, the University of Lille, the Initiative of Excellence of the University of Lille, and the European Metropolis of Lille for their funding and support of the R-CDP-24-004-C2EMPI project.}

\begin{document}

\title[Stable invariant measures in linear dynamics]{Stable invariant measures in linear dynamics}


\begin{abstract} 
We study the existence of stable invariant measures for operators and strongly continuous semigroups of operators on Banach spaces admitting either a dense bilateral backward orbit or a sufficiently rich family of eigenvectors. These invariant measures are realized as the distributions of stochastic integrals with respect to stable random measures. We also discuss invariant measures with other classes of distributions for such operators and semigroups.
\end{abstract}

\maketitle

\par\bigskip

\section{Introduction} \label{sectionIntroduction}
Invariant measures for linear operators on infinite-dimensional separable Banach spaces play a central role in linear dynamics. The first connections between ergodic theory and linear dynamics date back to the work of Bayart and Grivaux \cite{BayartGrivaux06}, who introduced the notion of a frequently hypercyclic operator. While hypercyclicity corresponds to the existence of a dense orbit, frequent hypercyclicity quantifies how often the iterates of a hypercyclic vector visit a given nonempty open set. Any linear operator on a separable Banach space having an ergodic invariant measure with full support is frequently hypercyclic, and it is well known that every operator satisfying the so called Frequent Hypercyclicity Criterion has an invariant measure with full support (see, e.g., \cite[Prop. 8.1]{BayartMatheron16}, \cite{MurilloPeris13}). 

\smallskip
The existence of sufficiently many unimodular eigenvectors often leads to the existence of invariant measures with full support. More precisely, if $T$ is a linear operator on a separable complex Banach space admitting a sufficiently rich family of eigenvectors associated to unimodular eigenvalues, then one can construct an ergodic invariant measure with full support for $T$ (see, e.g., \cite{BayartGrivaux06}, \cite{BayartGrivaux07}, \cite[Chapter 5]{BM09} and \cite{BayartMatheron16}). Note that the existence of non-degenerate Gaussian invariant measures for operators admitting a perfectly spanning set of eigenvectors associated to unimodular eigenvalues was first introduced by Flytzanis in \cite{Flytzanis95}.

\smallskip
It is well known that the set of hypercyclic vectors of an hypercyclic operator on a separable Banach space is dense $G_\delta$ (\cite[Thm. 1.2]{GodefroyShapiro98}).
By contrast, the set of frequently hypercyclic vectors of a frequently hypercyclic operator is meager in general (\cite[Theorem 5.1]{BonillaGrosseErdmann07}). As a consequence, Baire category methods are of limited use in the context of frequent hypercyclicity.
In counterpart, ergodic theory is a powerful tool in the context of frequent hypercyclicity since any operator admitting an ergodic invariant measure with full support has a set of full measure of frequently hypercyclic vectors, by Birkhoff's ergodic theorem. In particular, it shows that if $B_w$ is a weighted backward shift operator on $\ell_p(\Z_+)$ with $1 \leq p < \infty$ such that the weights satisfy $\sum_{n \geq 1} \frac{1}{\lvert w_1 \dotsm w_n \rvert^p} < \infty $, then the vector $\sum_{n \geq 0} \frac{g_n}{w_1 \dotsm w_n} e_n$ is almost surely frequently hypercyclic for $B_w$, where $(g_n)_{n \geq 0}$ is a sequence of standard Gaussian complex random variables (see \cite[Section 5.5.2]{BM09} or \cite[Section 7.1]{BayartMatheron16}). Bayart and Matheron also showed that the random vector $\sum_{n \geq 0} g_n \frac{z^n}{n!}$ is almost surely frequently hypercyclic for the derivation operator on the space of entire function. This was also proved by Nikula \cite{Nikula14} without using invariant measures (see also \cite{MouzeMunnier21}). Agneessens \cite{AGNEESSENS_2023} studied the existence of almost surely frequently hypercyclic vectors of the form $\sum_{n \in \Z}  X_n u_n$ for operators admitting a dense bilateral backward orbit $(u_n)_{n \in \Z}$, were $(X_n)_{n \in \Z}$ is a sequence of i.i.d random variables satisfying some conditions. In this paper, we discuss invariant measures for linear operators and strongly continuous semigroups of operators on Banach spaces.

\medskip

The first to study the dynamical properties of $\mathcal{C}_0$-semigroups of operators such as hypercyclicity and chaos are Desch, Schappacher and Webb \cite{DeschSchappacherWebb97}, even though Rolewicz \cite{Rolewicz69} was the first to observe the existence of a dense orbit under a $\mathcal{C}_0$-semigroup. Bermúdez, Bonilla, Conejero and Peris \cite{bermudezbonillaconejeroperis05} also studied mixing for semigroups.

\smallskip
As in the discrete-time setting, spectral properties of the generator of a strongly continuous semigroup of operators play an important role in the mixing and chaotic behavior of the semigroup. In particular, El Mourchid \cite{Mourchid06} showed that the existence of sufficiently many purely imaginary eigenvectors of the generator of a $\mathcal{C}_0$-semigroup leads to mixing and chaotic behavior of the semigroup (see also \cite[Theorem 7.32]{GEP}). Bayart and Matheron \cite[Thm. 8.3]{BayartMatheron16} showed that $\mathcal{C}_0$-semigroups of operators for which the generator has sufficiently enough purely imaginary eigenvalues are mixing in the Gaussian sense, and Murillo-Arcila and Peris \cite{MurilloPeris15} proved that any $\mathcal{C}_0$-semigroup of operators satisfying the Frequent Hypercyclicity Criterion has a strongly mixing invariant measure with full support. 

\smallskip

Making use of stochastic integrals with respect to complex Brownian motions, Chakir and El Mourchid \cite{ChakirMourchid18} constructed strongly mixing Gaussian measures with full support for operator semigroups whose generator admits a large family of purely imaginary eigenvalues.
In the same spirit, based to the work on random Gaussian integrals in \cite{vanNeervenWeis05} and generalizing his work on random frequently hypercyclic vectors for operators, Agneessens \cite{AGNEESSENS_2026} studied random frequently hypercylic vectors for $\mathcal{C}_0$-semigroups of operators using stochastic integration with respect to complex Brownian motions. 

\medskip
In this paper, we investigate the existence of random frequently hypercyclic vectors with symmetric stable distributions for operators and $\mathcal{C}_0$-semigroups of operators acting on infinite-dimensional separable Banach spaces. To this end, we consider stochastic integrals with respect to real and complex symmetric stable random measures on $\sigma$-finite measure spaces.
The present work is motivated by the contributions of Agneessens \cite{AGNEESSENS_2023,AGNEESSENS_2026} on random frequently hypercyclic vectors, as well as by the recent work of Mau and Privault \cite{MauPrivault}, who obtained criteria for an infinitely divisible measure on a Banach space to be weakly mixing or strongly mixing with respect to a linear operator in terms of codifference functionals. Most of the constructions in \cite{AGNEESSENS_2026} rely on covariance operators associated to Gaussian measures in order to produce Gaussian random frequently hypercyclic vectors for strongly continuous semigroups of operators. It is therefore natural to investigate whether analogous constructions can be carried out for other classes of infinitely divisible distributions. In particular, one may ask whether random frequently hypercyclic vectors with symmetric stable distributions can be constructed for operators and semigroups of operators, despite the possible absence of finite moments. This question constitutes the main motivation for the present paper.

\smallskip
Throughout the paper, $E$ denotes an infinite-dimensional separable Banach space. All random variables and random integrals are defined on a suitable probability space $(\Omega,\mathcal F,\mathbb P)$.
If $(J,\Sigma, m)$ is a measure space and if $0 < p < \infty$, we denote by $L^p(J;E)$ the space of measurable functions $f : J \to E$ such that $\int_J \lVert f(t) \rVert^p m(dt) < \infty$. For $0<\alpha\leq 2$, we denote by $c_0$ the constant
\[
c_0:=\frac{1}{2\pi}\int_0^{2\pi} |\cos(x)|^\alpha\,dx,
\]
which will appear repeatedly throughout the paper.

\bigskip
Recall that an operator $T$ on $E$ is said to be hypercyclic if there exists $x \in E$ such that the orbit $\{T^n x : n \geq 0\}$ is dense in $E$, and frequently hypercyclic if there exists $x \in X$ such that, for every nonempty open subset $V \subset E$, the set $\{n\in\mathbb N:\, T^n x\in V\}$ has positive lower density.
Similarly, one can define the notions of hypercyclic and frequently hypercyclic semigroups of operators using the Lebesgue measure to define the lower density of a subset of $[0,\infty)$.
We refer the reader to the two books \cite{EngelNagel06} and \cite[Chapter 7]{GEP} for background on operator semigroups, and to
\cite{BM09} and \cite{GEP} for background on linear dynamics.

\smallskip
We define the notion of a random measure on a $\sigma$-finite measure space $(J, \Sigma, m)$ as follows. A random measure on $(J,\Sigma,m)$ is a map $M$ on $\Sigma_0:= \{ A \in \Sigma : m(A) < \infty \} $ taking values in the space of real or complex random variables on $(\Omg, \mathcal{F},\P)$ such that for every pairwise disjoint sets $(A_k)_{1 \leq k \leq n} \subset \Sigma_0$, the random vectors $(M(A_k))_{1 \leq k \leq n}$ are independent and if for every pairwise disjoint sequence $(A_k)_{k \geq 1} \subset \Sigma_0$ such that $\displaystyle \bigcup_{k \geq 1} A_k \in \Sigma_0$, the series $\displaystyle \sum_{k \geq 1} M(A_k)$ converges in probability and 
$$
\displaystyle \sum_{k \geq 1} M(A_k) = M \left(\bigcup_{k \geq 1} A_k \right) \quad \textrm{a.s}.
$$
A standard way to construct a random measure with prescribed distribution is to apply Kolmogorov's existence theorem to a consistent family of finite-dimensional distributions. In the present work, we consider both real and complex stable random measures. Recall that an $E$-valued random variable $X$ is said to be $\alpha$-stable with $\alpha \in (0,2]$ if for any $a,b >0$ there exists $z \in E$ such that $a X^{(1)} + b X^{(2)}$ has the same distribution as $z + (a^\alpha + b^\alpha)^{1/\alpha} X$, where $X^{(1)}$ and $X^{(2)}$ are independent copies of $X$. If, moreover, $X$ and $-X$ have the same distribution, then we say that $X$ has a symmetric $\alpha$-stable ($S\alpha S$) distribution; if, furthermore, $E$ is complex and $\lambda X$ has the same distribution as $X$ for every $\lambda \in \T$, then we say that $X$ is isotropic. Note that the case $\alpha = 2$ corresponds to a Gaussian distribution. 

\smallskip
For $0<\alpha<2$, the characteristic function of an $E$-valued symmetric $\alpha$-stable random vector $X$ can be written by the Tortrat theorem (\cite{Tortrat}) as
\begin{align} \label{formuletortratstable}
    \E[e^{i \Re(\lan x^*, X \ran)}] = \int_E e^{i \Re(\lan x^*, z \ran)} \mu(dz) = \exp \left(- \int_{S_E} \lvert \Re(\lan x^*, z \ran) \rvert^\alpha \xi(dz) \right)
\end{align}
where $\xi$ is a symmetric finite measure concentrated on the unit sphere $S_E$ of $E$. Such a measure is unique, see Linde \cite[Thm. 6.4.4]{Linde86}, Ledoux--Talagrand \cite[Corollary 5.5]{LedouxTalagrand91} or Woyczyński \cite[p. 6]{Woyczynski19} for the Banach case and Samorodnitsky and Taqqu  \cite[Thm. 2.4.3]{SamorodnitskyTaqqu94} for the complex case. We call the measure $\xi$ the spectral measure of the symmetric $\alpha$-stable random vector $X$. For $0<\alpha<2$, a complex symmetric $\alpha$-stable random variable $X$ is isotropic (i.e.\ rotationally invariant) if and only if its (necessarily unique) spectral measure is the uniform probability measure on the unit circle (see \cite[Section 2.6]{SamorodnitskyTaqqu94}). For a complex symmetric $2$-stable random variable, the expression (\ref{formuletortratstable}) also holds but the measure $\xi$ is not necessarily unique (see \cite[p. 76, Remark 3]{SamorodnitskyTaqqu94}); nevertheless, any symmetric $2$-stable complex random variable admitting the uniform probability measure on the unit circle as a spectral measure is isotropic (\cite[p. 87]{SamorodnitskyTaqqu94}). In this case, its real and imaginary parts are independent and identically distributed centered Gaussian random variables.

\smallskip
In order to construct random frequently hypercyclic vectors for operators or strongly continuous semigroups of operators, we will need certain geometric conditions on the underlying Banach space. Let $E$ be a real or complex separable Banach space. We say that $E$ is of type $p$ (with $1 \leq p \leq 2$) if for every sequence $(x_n)_{n \geq 1} \in \ell_p(E)$, the series $\sum_{n \geq 1} \varepsilon_n x_n$ is almost surely convergent in $E$, where $(\varepsilon_n)_{n \geq 1}$ is a sequence of i.i.d random variables with distribution $\tfrac{1}{2}(\delta_1 + \delta_{-1})$. By Kahane's inequality, the Banach space $E$ is of type $p$ if and only if for some (hence every) finite $0<q < \infty$, there exists a constant $C_{p,q}$ such that, for all finite sequence $(x_n) \subset E$,
$$
\left(\E   \left \lVert \sum_n \varepsilon_n x_n \right \rVert^q  \right)^{1/q} \leq C_{p,q} \left(\sum_n \lVert x_n \rVert^p \right)^{1/p}.
$$
We say that $E$ is of stable-type $p$ (with $0<p\leq 2$) if for every sequence $(x_n)_{n \geq 1} \in \ell_p(E)$, the series $\sum_{n \geq 1} \theta_n^{(p)} x_n$ is almost surely convergent in $E$, where $( \theta_n^{(p)})_{n \geq 1}$ is a sequence of i.i.d standard symmetric $p$-stable random variables (i.e, real random variables satisfying $\E(e^{it \theta_n^{(p)}}) = \exp(-\lvert t \rvert^p)$ for every $t \in \R$). As a consequence of a theorem of Hoffmann-Jørgensen (see, for instance, \cite[Prop. 3.4.5]{Linde86}, \cite[Thm. 12.4]{Schwartz81} or \cite[Remark 6.5.1]{Woyczynski19}), $E$ is of stable-type $p$ if and only if for some (hence every) $q<p$ there exists a constant $C_{p,q}$ such that, for all finite sequence $(x_n) \subset E$,
$$
\left (\E \left \lVert \sum_n \theta_n^{(p)} x_n \right \rVert^q \right)^{1/q} \leq C_{p,q} \left(\sum_n \lVert x_n \rVert^p \right)^{1/p}.
$$
Note that if $E$ is of type $1 \leq p \leq 2$ (resp. of stable-type $0<p\leq 2$), then it is also of type $1\leq q\leq p$ (resp. of stable-type $0 < q \leq p$).
If $E$ is of stable-type $1 \leq p \leq 2$, then it is of type $p$ and if $E$ is of type $1 \leq p \leq 2$, then $E$ is of stable type $q$ for any $0<q<p$ but not of stable-type $p$ since $\theta_n^{(p)}$ does not have $p$-moment. Finally, $E$ is of type $2$ if and only if it is of stable-type $2$. We refer to \cite[Chapter 5]{LiQueffelec18}, \cite[Section 3.5]{Linde86}, \cite[Chapter III]{Schwartz81} and \cite[Sections 6.2 and 6.5]{Woyczynski19} for more details.

\section{Random integrals with respect to symmetric stable random measures}

In this section, we recall the constructions of random integrals of scalar functions with respect to a real or complex random measure on a $\sigma$-finite measure space $(J,\Sigma, m)$ based on Samorodnitsky and Taqqu \cite[Chapters 3 and 6]{SamorodnitskyTaqqu94}. After that, we construct a Banach valued random integrals with respect to a real or complex symmetric stable random measure with the aim of producing random frequently hypercyclic vectors for $\mathcal{C}_0$-semigroups of operators. Under some assumptions on the Banach space $E$ considered, the random frequently hypercyclic vectors that we will construct have a symmetric stable distribution on $E$ (rotationally invariant if $E$ is complex).

\subsection{Integrals with respect to symmetric real stable random measures}
A random measure $M$ on $(J,\Sigma,m)$ is a symmetric real $\alpha$-stable random measure with control measure $m$ ($\alpha \in (0,2]$) if it is a random measure satisfying 
$$
\E(e^{i tM(A)}) = \exp(-m(A) \cdot \lvert t \rvert^\alpha) \quad \textrm{for every $A \in \Sigma_0$ and $t \in \R$}.
$$
In other words, a random measure is real symmetric $\alpha$-stable if $M(A)$ is a real symmetric $\alpha$-stable random variable with parameter $m(A)^{1/\alpha}$ for every $A \in \Sigma_0$. The existence of real symmetric random measures follows from Kolmogorov's existence theorem, see Samorodnitsky and Taqqu \cite[Chapter 3, Section 3.3]{SamorodnitskyTaqqu94}. 

\smallskip
For a function $f : J \to \R$, the random integral of $f$ with respect to a real symmetric $\alpha$-stable random measure is constructed as follows. For a simple function $f = \displaystyle \sum_{j=1}^n a_j \indic_{A_j}$, where $a_j \in \R$ and the $A_js$ are disjoint sets belonging to $\Sigma_0$, we set
$$
\int_J f(t) M(dt) := \sum_{j=1}^n a_j M(A_j).
$$
We deduce that, for such a simple function $f$, we have
\begin{align} \label{eqFonccaractfoncreellemesurerellle}
\E(e^{i t \int_J f(x) M(dx)}) = \exp \left(- \int_J \lvert f(x) \rvert^\alpha m(dx) \cdot  \,\lvert t \rvert^\alpha \right) \quad \textrm{for every $t \in \R$}.
\end{align}
For a general function $f \in L^\alpha(J;\R)$, there exists a sequence $(f_n)_{n \geq 1}$ of simple functions converging almost everywhere to $f$ such that $\lvert f_n \rvert \leq \lvert f \rvert$ for every $n \geq 1$. It can be shown that the sequence $(\int_J f_n dM)_{n \geq 1}$ converges in probability and we thus define $\int_J f(x) M(dx)$ by 
$$
\int_J f(x) M(dx) := \displaystyle \lim_{n \to \infty} \int_J f_n(x) M(dx),
$$
where the limit holds for the convergence in probability and does not depend on the choice of the approximating sequence of step functions $(f_n)_{n \geq 1}$. Since convergence in probability implies convergence in distribution, the random variable $\int_J f \, dM$ has a real symmetric $\alpha$-stable distribution and (\ref{eqFonccaractfoncreellemesurerellle}) holds again. 
Moreover, the two following properties of the random integral hold (see \cite[Prop. 3.5.1 and Thm. 3.5.3]{SamorodnitskyTaqqu94}). 

\smallskip
\begin{proposition} \label{propintegralestochstablemesurereelle}
Let $\alpha \in (0,2]$ and let $M$ be a real symmetric $\alpha$-stable random measure on $(J,\Sigma,m)$ with $m$ as control measure. 
    \begin{enumerate}
        \item If $f \in L^\alpha(J;\R)$ and $(f_j)_{j \geq 1} \subset L^\alpha(J;\R)$, then $\int_J f_j \, dM$ converges in probability to $\int_J f \, dM$ if and only if $\int_J \lvert f_j -f \rvert^\alpha dm$ converges to $0$.
        \smallskip
        \item If $\alpha <2$ and if $f_1, f_2 \in L^\alpha(J;\R)$, then $\int_J f_1 \, dM$ and $\int_J f_2 \, dM$ are independent if and only if $f_1 \cdot f_2 \equiv 0$ $m$-almost everywhere on $J$.
        \smallskip
        \item If $\alpha = 2$ and if $f_1, f_2 \in L^2(J;\R)$, then $\int_J f_1 \, dM$ and $\int_J f_2 \, dM$ are independent if and only if $\int_J f_1(x) f_2(x) m(dx) = 0$.
    \end{enumerate}
\end{proposition}

\smallskip
\begin{comment}
If $f : T \to \C$ is a function in $L^\alpha(T;\C)$ and $M$ is a real symmetric $\alpha$-stable random measure with $0< \alpha \leq 2$, we define $\int_T f \, dM$ by setting
\begin{align}
    \int_T f \,dM := \int_T \Re(f)\, dM + i \int_T \Im(f) dM.
\end{align}
Note that in this case, $\int_T f\, dM$ has a complex symmetric $\alpha$-stable distribution with characteristic function
\begin{align}
    \E(e^{i( \theta_1 \int_T \Re(f)\, dM + \theta_2 \int_T \Im(f) \, dM)}) = \exp(-\int_T \lvert \theta_1 \Re(f) + \theta_2 \Im(f) \rvert^\alpha dm), \; \theta_1, \theta_2 \in \R.
\end{align}

\smallskip
\end{comment}

We now move to the construction of random integrals of Banach valued functions with respect to a symmetric real stable random measure. 
Let $E$ be a real separable Banach space and let $M$ be a real symmetric $\alpha$-stable random measure on $(J,\Sigma,m)$ with control measure $m$. A function $f : J \to E$ is weakly stochastically integrable with respect to $M$ if for every $x^* \in E^*$, the function $t \mapsto \lan x^*, f(t) \ran$ belongs to $L^\alpha(J;\R)$ and if there exists a random vector $Y : \Omg \to E$ such that for every $x^* \in E^*$, 
\begin{align}
    \lan x^*, Y \ran = \int_J \lan x^*, f \ran \, dM \quad \textrm{a.s},
\end{align}
where the real integral was just defined above. 
In this case, we write $Y= \int_J f\, dM$. Note that $Y$ is uniquely determined. Moreover, the following properties easily follow.

\smallskip

\begin{proposition} \label{propweaklystochintrealrandommeasure}
Let $E$ be a real Banach space and let $M$ be a real symmetric stable random measure on $(J,\Sigma,m)$ with control measure $m$.
If $f,g : J \to E$ are weakly stochastically integrable with respect to $M$, if $a,b \in \R$ and if $T : E \to E$ is an operator on $E$, then $af + b g$ and $T f$ are weakly stochastically integrable with respect to $M$ and we have $$\int_J (a f + b g ) \, dM = a \int_J f \, dM + b \int_J f \, dM$$
and
$$
\int_J Tf \, dM = T\int_J f\, dM.
$$      
\end{proposition}

\medskip

In the case where $f : J \to E$ is weakly stochastically integrable with respect to the real symmetric random measure $M$ on $(J,\Sigma,m)$, we can easily compute the characteristic function of $\int_J f \, dM$. It is given for every $x^* \in E^*$ by
\begin{align} \label{fonccaractrandomintBanachMreelle}
    \E(e^{i \lan x^*, \int_J f\, dM \ran}) &= \exp \left(-\int_J \lvert \lan x^*, f(t) \ran \rvert^\alpha m(dt) \right).
\end{align}
\smallskip
With (\ref{fonccaractrandomintBanachMreelle}), we can deduce the spectral measure of the distribution of $\int_J f \, dM$ in the case where $f \in L^\alpha(J;E)$ and $M$ is a symmetric real $\alpha$-stable random measure on $(J,\Sigma,m)$ when $0 < \alpha <2$.
Let $\xi = \int_J \frac{1}{2} \lVert f(t) \rVert^\alpha (\delta_{\frac{f(t)}{\lVert f(t) \rVert}} + \delta_{-\frac{f(t)}{\lVert f(t) \rVert}}) m(dt)$.  The measure $\xi$ is finite, symmetric and concentrated on $S_E$, and satisfies
\begin{align} \label{spectralrepresentationrealBanach}
    \E(e^{i \lan x^*, \int_J f\, dM \ran}) = \exp \left(-\int_{S_E} \lvert \lan x^*, z \ran \rvert^\alpha \xi(dz) \right)
\end{align}
for every $x^* \in E^*$. In other words, $\xi$ is the spectral measure of the distribution of $\int_J f \, dM$ when $f : J \to E$ belongs to $L^\alpha(J;E)$ and when the random measure $M$ is real $S\alpha S$ with control measure $m$.

\smallskip

Under some geometric assumptions on the real Banach space $E$, any function in $L^\alpha(J;E)$ is weakly integrable with respect to the real symmetric stable random measure.
Let $E$ be a real Banach space of stable-type $\alpha \in (0,2)$. Let $M$ be a real symmetric $\alpha$-stable random measure on $(J,\Sigma,m)$ with control measure $m$ and let $f : J \to E$ be in $L^\alpha(J;E)$. If $f = \displaystyle \sum_{j=1}^n x_j \indic_{A_j}$ is a simple function with $x_j \in E$ and the sets $A_j$ are disjoint and belong to $\Sigma_0$, we set
\begin{align}
    \int_J f(t) M(dt) := \sum_{j=1}^n x_j M(A_j),
\end{align}
which is a symmetric $\alpha$-stable $E$-valued random vector. Let us notice that the map $f \mapsto \int_J f\, dM$ is linear on the space of simple functions in $L^\alpha(J;E)$ and takes values in $L^q(\Omega;E)$ for every $q<\alpha$. Since $E$ is of stable-type $\alpha$, for every $q<\alpha$, we have
\begin{align*}
    \left \lVert \int_J f \, dM \right \rVert_q &= \left (\E \left \lVert \sum_{j=1}^n x_j M(A_j) \right \rVert^q  \right)^{1/q} \\
    &\leq C  \, \left( \sum_{j=1}^n m(A_j) \lVert x_j \rVert^\alpha \right)^{1/\alpha} \\
    &\leq C \, \left(\int_J \lVert f(t) \rVert^\alpha m(dt) \right)^{1/\alpha},
\end{align*}
where $C$ is a constant depending on $\alpha$ and $q$.
Since the simple functions are dense in $L^\alpha(J;E)$, we can extend the definition of $\int_J f \, dM$ to every function $f \in L^\alpha(J;E)$ by setting 
\begin{align} \label{definitionintegraleEdetypestablealpha}
    \int_J f\, dM := \lim_{n \to \infty} \int_J f_n \, dM
\end{align}
for every sequence of simple functions $(f_n)_{n \geq 1} \subset L^\alpha(J;E)$ converging to $f$, where the convergence holds in $L^q(\Omega;E)$ with $q<\alpha$ (and hence in probability). We thus obtain an $E$-valued random vector $\int_J f \, dM$ for every function $f \in L^\alpha(J;E)$ satisfying
\begin{align} \label{inequationintegralesBanachtypestable}
    \left \lVert \int_J f\, dM \right \rVert_q \leq C \left(\int_J \lVert f(t) \rVert^\alpha m(dt) \right)^{1/\alpha}
\end{align}
for every $0<q<\alpha$. Moreover, the map $f \mapsto \int_J f \, dM$ is linear on $L^\alpha(J;E)$ and takes values in $L^q(\Omg;E)$ for every $q<\alpha$. We refer to  \cite[Section 6.10]{Woyczynski19} for more details. The following observation is easy.

\smallskip

%We can also compute the characteristic function of $\int_T f \, dM$. Indeed, let $x^* \in E^*$ and let $f_n = \sum_j x_j^n \indic_{A_j^n}$ be a sequence of simple functions converging to $f$ in $L^\alpha(T;E)$. Since convergence in probability implies convergence in distribution, we obtain
%\begin{align*}
%    \E(e^{i \Re (\lan x^*, \int_T f\, dM\ran)}) &= \lim_{n \to \infty} \E(\exp(i \sum_j \Re(\lan x^*, x_j^n \ran) M(A_j))) \\
%    &=\lim_{n \to \infty} \prod_j \exp(-m(A_j^n) \lvert \Re (\lan x^*, x_j^n \ran) \rvert^\alpha )) \\
%    &= \lim_{n \to \infty} \exp(-\int_T \lvert \Re (\lan x^*, f_n(t) \ran) \rvert^\alpha m(dt)) \\
%    &= \exp(-\int_T \lvert \Re (\lan x^*, f(t) \ran) \rvert^\alpha m(dt)),
%\end{align*}
%where we used the fact that $M$ is a real random measure. This shows in particular that $\int_T f \, dM$ is symmetric $\alpha$-stable on $E$.
%\smallskip

\begin{lemma} \label{LemmeMesurealéatoirestablereelleintegrfaibleetforte}
    Suppose that $E$ is a real Banach space of stable type $\alpha \in (0,2]$. Let $M$ be a real symmetric $\alpha$-stable random measure on $(J,\Sigma,m)$ with control measure $m$ and let $f \in L^\alpha(J;E)$. Then $f$ is weakly stochastically integrable with respect to $M$.
\end{lemma}

\smallskip
\begin{proof}
    Let $(f_n)_{n \geq 1}$ be a sequence of step functions converging to $f$ in $L^\alpha(J;E)$. Let $x^* \in E^*$. Let us notice that $\lan x^*, \int_J f_n \, dM \ran = \int_J \lan x^*, f_n \ran dM$ for every $n \geq 1$. Since $\lan x^*, f_n \ran$ converges to $\lan x^*, f \ran$ in $L^\alpha(J; \R)$, we obtain that $\lan x^*, \int_J f_n \, dM \ran $ converges to $\int_J \lan x^*, f \ran \, dM$ in probability (by Proposition \ref{propintegralestochstablemesurereelle}). Moreover, $\lan x^*, \int_J f_n \, dM \ran$ converges to $\lan x^*, \int_J f \, dM \ran$ in probability by (\ref{inequationintegralesBanachtypestable}). Thus $\lan x^*, \int_J f \, dM \ran = \int_J \lan x^*, f \ran \, dM$ almost surely.
\end{proof}

\subsection{Random integrals with respect to complex isotropic stable random measures}
A random measure $M$ on a measurable space $(J,\Sigma)$ is a complex $\alpha$-stable random measure with circular control measure $k$ if it is a complex random measure such that, for every $A \in \Sigma_0 := \{ B \in \Sigma : k(B \times \mathbb{S}_2) < \infty \}$, the random variables $M^{(1)}(A) := \Re(M(A))$ and $M^{(2)}(A) := \Im(M(A))$ are jointly $S\alpha S$ with spectral measure $k(A \times \cdot)$, where $k$ is a measure on the product space $J \times \mathbb{S}_2$ such that for every $A \in \Sigma$ with $k(A \times \mathbb{S}_2) < \infty$, $k(A \times \cdot )$ is a finite symmetric Borel measure on $\mathbb{S}_2$. Here, independence of the complex-valued random variables $M(A_1), \dotsc, M(A_n)$ means independence of the vectors $\begin{pmatrix}
M^{(1)}(A_1)\\
M^{(2)}(A_1)
\end{pmatrix}, \dotsc, \begin{pmatrix}
M^{(1)}(A_n)\\
M^{(2)}(A_n)
\end{pmatrix} $. We suppose in what follows that the complex stable random measures are \textbf{isotropic} (i.e, rotationally invariant), meaning that the circular control measure has the form $k = m  \gamma$, where $\gamma$ is the uniform probability measure on the unit sphere $\mathbb{S}_2$ of $\C$ and $m$ is a $\sigma$-finite measure on $(J,\Sigma)$ called the control measure. 

\smallskip

We define the random integral $\int_J f \, dM$ for $f : J \to \C$ and $M$ a complex isotropic $\alpha$-stable random measure with circular control measure $k$ exactly as in the case of a real $S\alpha S$ random measure: for a simple function $f = \displaystyle \sum_{j=1}^n a_j \indic_{A_j}$ where the $a_js$ are complex and the $A_js$ are disjoint sets in $\Sigma_0$, we set
$$
\int_J f \, dM := \sum_{j=1}^n a_j M(A_j).
$$
For a general function $f \in L^\alpha(J;\C)$, we choose a sequence $(f_n)_{n \geq 1}$ of complex simple functions converging to $f$ almost everywhere such that $\lvert f_n \rvert \leq \lvert f \rvert$ for every $n \geq 1$. It can be shown again that the sequence $(\int_J f_n \, dM)_{n \geq 1}$ converges in probability and we thus set
$$
\int_J f \, dM := \lim_{n \to \infty} \int_J  f_n \, dM,
$$
where the limit holds for the convergence in probability and does not depend on a particular choice of approximating sequence $(f_n)_{n \geq 1}$. The map $f \mapsto \int_J f\, dM$ is linear on $L^\alpha(J;\C)$. Moreover, for every function $f \in L^\alpha(J;\C)$, by \cite[Prop. 6.2.4]{SamorodnitskyTaqqu94}, the random variable $\Re(\int_J f \, dM)$ is symmetric $\alpha$-stable with characteristic function
\begin{align}
    \E(e^{it \Re(\int_J f \, dM)}) = \exp \left(-\int_J \left(\int_{\mathbb{S}_2} \lvert u_1 \Re(f(x)) - u_2 \Im(f(x)) \rvert^\alpha \gamma(du) \right) m(dx) \cdot \lvert t \rvert^\alpha \right),\; t \in \R.
\end{align}
Note that this characteristic function can also be written as
\begin{align} \label{fonccarfonccomplexmesureisotropecomplex}
    \E(e^{it \Re(\int_J f \, dM)}) = \exp \left(- c_0\int_J \lvert f(x) \rvert^\alpha m(dx) \cdot \lvert t \rvert^\alpha \right), \; t \in \R,
\end{align}
where $c_0 := \frac{1}{2\pi} \int_0^{2 \pi} \lvert \cos \phi \rvert^\alpha d\phi $, and the random variables $\Re(\int_J f \, dM)$ and $\Im(\int_J f \, dM)$ have the same distribution, see also \cite[Thm. 6.3.1]{SamorodnitskyTaqqu94}.

\smallskip
If $\alpha = 2$, then for every function $f \in L^2(J;\C)$, we have the following equality
\begin{align} \label{propisometryintcomplexemesurecomplexe}
    \E \left( \left \lvert \int_J f \, dM \right \rvert^2 \right) = 4 c_0 \int_J \lvert f(t) \rvert^2 m(dt).
\end{align}
Indeed, the relation (\ref{propisometryintcomplexemesurecomplexe}) is satisfied for a simple function $f = \displaystyle \sum_{j=1}^n a_j M(A_j)$ since in this case we have
\begin{align*}
    \E \left( \left \lvert \int_J f \, dM \right \rvert^2 \right) &= \sum_{k,l = 1}^n a_k \, \overline{a_l} \, \E(M(A_k) \overline{M(A_l)}) \\
    &= 4 c_0 \sum_{k=1}^n \lvert a_k \rvert^2 m(A_k) \\
    &= 4 c_0 \int_J \lvert f(t) \rvert^2 m(dt).
\end{align*}

\smallskip
Similarly to the case of a real symmetric stable random measure, the following properties hold. 

\smallskip

\bpr \label{propintscalairemesurealéacomplexestable}
Let $\alpha \in (0,2]$ and let $M$ be a complex isotropic $\alpha$-stable random measure on $(J,\Sigma)$ with control measure $m$.
\begin{enumerate}
    \item For $f, f_j \in L^\alpha(J;\C)$, the sequence $\int_J f_j \, dM$ converges in probability to $\int_J f \, dM$ if and only if $\int_J \lvert f_j - f \rvert^\alpha dm$ converges to $0$.
    \smallskip
    \item For $f,g \in L^\alpha(J;\C)$ and $\alpha<2$, the random variables $\int_J f \, dM$ and $\int_J g \, dM$ are independent if and only if $f \cdot g \equiv 0$ $m$-a.e.
    \item If $\alpha = 2$ and if $f,g \in L^2(J;\C)$, the random variables $\int_J f \, dM$ and $\int_J g \, dM$ are independent if and only if $\int_J f(x) \overline{g(x)} m(dx) = 0$.
\end{enumerate}
\epr

If $E$ is a complex separable Banach space and if $M$ is a complex isotropic $\alpha$-stable random measure on $(J,\Sigma)$ with control measure $m$, we define the notion of a weakly stochastically integrable function with respect to the random measure $M$ exactly as in the case of a real symmetric stable random measure on $(J,\Sigma,m)$. In particular, Proposition \ref{propweaklystochintrealrandommeasure} holds again for a complex isotropic stable random measure, and in this case the characteristic function is given for every $x^* \in E^*$ by 
\begin{align}
\notag    \E(e^{i \Re(\lan x^*, \int_J f\, dM \ran)}) &= \exp \left(-\int_J \left (\int_{\mathbb{S}_2} \lvert u_1 \Re(\lan x^*, f(t) \ran) - u_2 \Im(\lan x^*, f(t) \ran)  \rvert^\alpha \,  \gamma(du) \right) m(dt) \right) \\ \label{fonccaractintaleaMcomplexe}
    &= \exp \left(-\int_J \left( \int_{\mathbb{S}_2} \lvert \Re(\lan x^*, u f(t) \ran)\rvert^\alpha \gamma(du) \right) m(dt) \right). 
\end{align}
Note that (\ref{fonccaractintaleaMcomplexe}) can also be written using (\ref{fonccarfonccomplexmesureisotropecomplex}) as
\begin{align} \label{fonccarrandomintegralsmesurecomplexeisotrope}
    \E(e^{i \Re(\lan x^*, \int_J f\, dM \ran)}) &= \exp \left(- c_0 \int_{J} \lvert \lan x^*, f(t) \ran \rvert^\alpha m(dt) \right),
\end{align}
with $c_0 = \frac{1}{2\pi}\int_0^{2\pi} \lvert \cos \phi \rvert^\alpha d\phi$. We can also identify the spectral measure of the distribution of $\int_J f \, dM$ when $f \in L^\alpha(J;E)$. Indeed, let $\xi = \int_{J \times \mathbb{S}_2} \frac{1}{2} \lVert f(t) \rVert^\alpha (\delta_{u \frac{f(t)}{\lVert f(t) \rVert}} + \delta_{- u\frac{f(t)}{\lVert f(t) \rVert}}) k(dt,du)$. The measure $\xi$ is finite, symmetric and concentrated on $S_E$, and
\begin{align} \label{mesurespectralemesuremustablecomplexe}
    \E(e^{i \Re(\lan x^*, \int_J f\, dM \ran)}) = \exp \left(-\int_{S_E} \lvert \Re( \lan x^*, z \ran) \rvert^\alpha \xi(dz) \right)
\end{align}
for every $x^* \in E^*$. In other words, $\xi$ is the spectral measure of the distribution of $\int_J f \, dM$ when $f : J \to E$ belongs to $L^\alpha(J;E)$ and when the random measure $M$ is complex isotropic $\alpha$-stable on $(J,\Sigma,m)$.

\smallskip

Finally, if $E$ is a complex separable Banach space of stable type $\alpha \in (0,2]$, if  $M $ is a complex isotropic $\alpha$-stable random measure on $(J,\Sigma,m) $ and if $f : J \to E$ belongs to $L^\alpha(J;E)$, we can define the random integrals $\int_J f \, dM$ exactly as in the case of a real symmetric stable random measure on $(J,\Sigma,m)$ since the inequality (\ref{inequationintegralesBanachtypestable}) also holds in this context. Indeed, suppose first that $f = \displaystyle \sum_{j=1}^n x_j \indic_{A_j}$ is a simple function. Then for $1 \leq  q<\alpha$ we just use the triangle inequality to obtain
\begin{align*}
    \left \lVert \int_J f \, dM \right\rVert_q &= \left \lVert \sum_{j=1}^n x_j M(A_j) \right\rVert_q \\
    &\leq \left \lVert \sum_{j=1}^n x_j \Re(M(A_j)) \right \rVert_q + \left \lVert \sum_{j=1}^n x_j \Im(M(A_j)) \right \rVert_q \\
    &\leq c \, \cdot \, \left (\sum_{j=1}^n m(A_j) \lVert x_j \rVert^\alpha \right)^{1/\alpha} \\
    &\leq c \, \cdot \left (\int_J \lVert f(t) \rVert^\alpha m(dt) \right)^{1/\alpha}
\end{align*}
where the constant $c$ only depends on $\alpha,q$ and on $c_0 = \frac{1}{2\pi} \int_0^{2\pi} \lvert \cos \phi \rvert^\alpha d\phi$. For $0<q<1$ and $q<\alpha$, we use that for every $a,b \geq 0$ we have $(a+b)^q \leq a^q + b^q$, and hence for every functions $f_1, f_2 \in L^q(\Omg;E)$ we have $\lVert f_1 + f_2 \rVert_q^q \leq \lVert f_1 \rVert_q^q + \lVert f_2 \rVert_q^q $, showing that the inequality again holds with a different constant when $f$ is a simple 
function. Since the simple functions are dense in $L^\alpha(J;E)$, we can again extend the definition of $\int_J f\, dM$ to every function $f \in L^\alpha(J;E)$ similarly to (\ref{definitionintegraleEdetypestablealpha}), and the random integral $\int_J f\,dM$ satisfies the same properties as for a symmetric stable real random measure on $(J,\Sigma,m)$. Moreover, Lemma \ref{LemmeMesurealéatoirestablereelleintegrfaibleetforte} also holds in the case of a complex isotropic stable random measure.
\smallskip

\begin{lemma} \label{LemmeMesurealéatoirestablecomplexeintegrfaibleetforte}
    Suppose that $E$ is a complex separable Banach space of stable type $\alpha \in (0,2]$. Let $M$ be a complex isotropic $\alpha$-stable random measure on $(J,\Sigma)$ with control measure $m$ and let $f \in L^\alpha(J;E)$. Then $f$ is weakly stochastically integrable with respect to $M$.
\end{lemma}

\smallskip
\begin{proof}
    The proof follows exactly the same lines as that of Lemma \ref{LemmeMesurealéatoirestablereelleintegrfaibleetforte}, using Proposition \ref{propintscalairemesurealéacomplexestable} this time. 
\end{proof}

\section{Stable invariant measures for $\mathcal{C}_0$-semigroups of operators}

\subsection{Main results}

Here, we denote by $m$ the Lebesgue measure on $\R$ and $(J,\mathcal{B}_J)$ denotes an interval equipped with its Borel subsets.
The aim of this subsection is to prove two general results on the existence of stable invariant measures with full support for strongly continuous semigroups of operators. The first one is the following. 

\smallskip

\bth \label{thm1espacereeloucomplexemesurealéaréelleoucomplexe}
Let $E$ be a real or complex infinite dimensional separable Banach space of stable type $\alpha \in (0,2]$. Let $(T_s)_{s \geq 0}$ be a $\mathcal{C}_0$-semigroup of operators on $E$. Suppose that there exists a family $(u_t)_{t \in \R}$ in $E$ with $t \mapsto u_t \in L^\alpha(\R;E)$ such that $T_s u_t = u_{t-s}$ for every $t \in \R$ and $s \geq 0$. Assume moreover that $\textrm{span}\{ u_t : t \in \R \setminus N \}$ is dense in $E$ for every Borel subset $N \subset \R$ with $m(N) = 0$. Let $M$ be either a real symmetric $\alpha$-stable random measure on $(\R,\mathcal{B}_\R)$ if $E$ is real, or a complex isotropic $\alpha$-stable random measure on $(\R,\mathcal{B}_\R)$ if $E$ is complex, with the Lebesgue measure $m$ as control measure. Then the distribution of the $E$-valued random vector
\[
v:= \int_{\R} u_t \, M(dt)
\]
is invariant and strongly mixing with full support under $(T_s)_{s \geq 0}$, and $v$ is a $S\alpha S$ random vector in $E$. 
\eth

\smallskip
From Theorem \ref{thm1espacereeloucomplexemesurealéaréelleoucomplexe}, we deduce the following result.

\smallskip

\bco
Let $E$ be a real or complex infinite dimensional separable Banach space of stable type $\alpha \in (0,2]$. Let $(T_s)_{s \geq 0}$ be a $\mathcal{C}_0$-semigroup of operators on $E$. Suppose that there exists a family $(u_t)_{t \in \R}$ in $E$ with $t \mapsto u_t \in L^\alpha(\R;E)$ such that $T_s u_t = u_{t-s}$ for every $t \in \R$ and $s \geq 0$. Assume moreover that $\textrm{span}\{ u_t : t \in \R \setminus N \}$ is dense in $E$ for every Borel subset $N \subset \R$ with $m(N) = 0$. Let $M$ be either a real symmetric $\alpha$-stable random measure on $(\R,\mathcal{B}_\R)$ if $E$ is real, or a complex isotropic $\alpha$-stable random measure on $(\R,\mathcal{B}_\R)$ if $E$ is complex, with the Lebesgue measure $m$ as control measure.
Then the random vector 
$$
v = \int_{\R} u_t M(dt)
$$
is almost surely frequently hypercyclic for the semigroup $(T_s)_{s \geq 0}$. Moreover, for each $t >0$, the operator $T_t$ is topologically strongly mixing.
\eco

\smallskip

We first focus on the proof of Theorem \ref{thm1espacereeloucomplexemesurealéaréelleoucomplexe} and start by proving multiple lemmas. The first one is the analogue of \cite[Lemma 2.6]{AGNEESSENS_2026} for symmetric stable distributions.

\smallskip

\blm \label{lemmeindependanceintegrales}
Let $E$ be a real or complex separable Banach space and let $J$ be an interval of $\R$. Let $M$ be a real symmetric stable random measure on $(J,\mathcal{B}_J)$ if $E$ is real, or a complex isotropic stable random measure on $(J,\mathcal{B}_J)$ if $E$ is complex, with the Lebesgue measure $m$ as control measure. Let $A,B$ be two disjoint Borel subsets of $J$ and let $f, g : J \to E$ be weakly stochastically integrable with respect to $M$. Then the random vectors $\int_A f \, dM$ and $\int_B g \, dM$ are independent.
\elm

\smallskip

\bpf
By \cite[p. 23]{BuldyginSolntsev97}, it suffices to show that for every $x^*, y^* \in E^*$, the scalar random variables $\lan x^*, \int_A f \, dM \ran$ and  $\lan y^*, \int_B g \, dM \ran$ are independent. But this is clear from Propositions \ref{propintegralestochstablemesurereelle} and \ref{propintscalairemesurealéacomplexestable} since $\lan x^*, \int_A f \, dM \ran = \int_J \lan x^*, f \ran \indic_A \, dM$, $\lan x^*, \int_B g \, dM \ran = \int_J \lan x^*, g \ran \indic_B\, dM$ and the functions $\lan x^*, f \ran \indic_A$ and $\lan x^*, g \ran \indic_B$ have disjoint supports.
\epf

The second lemma is the analogue of \cite[Lemma 2.5]{AGNEESSENS_2026} and directly follows from the expression of the characteristic function.

\smallskip
\blm \label{lemmeloiintegralesintervallestranslatés}
Let $E$ be a real or complex separable Banach space and let $J$ be an interval of $\R$. Let $M$ be either a real symmetric stable random measure on $(J,\mathcal{B}_J)$ if $E$ is real, or a complex isotropic stable random measure on $(J,\mathcal{B}_J)$ if $E$ is complex, with the Lebesgue measure $m$ as control measure. Let $f : J \to E$ be weakly stochastically integrable with respect to $M$. Then for any $s \in \R$, the function $f(\cdot-s)$ is weakly stochastically integrable with respect to $M$ and the random vectors $\int_{J+s} f(t-s) M(dt)$ and $\int_J f(t) M(dt)$ have the same distribution.
\elm

\smallskip

\bpf
We show that the random vectors have the same characteristic function. Suppose for example that the Banach space is real. Let $x^* \in E^*$. Then using (\ref{fonccaractrandomintBanachMreelle}), we have
\begin{align*}
    &\E(\exp({i \lan x^*, \int_{J+s} f(t-s) M(dt) \ran))} \\ 
    &= \exp(- \int_{J+s} \lvert \lan x^*, f(t-s) \ran \rvert^\alpha m(dt)) \\
    &= \exp(-\int_J \lvert \lan x^*, f(t) \ran \rvert^\alpha m(dt)) \\
    &= \E(\exp({i \lan x^*, \int_{J} f(t) M(dt) \ran))}.
\end{align*}
If the Banach space $E$ is complex, we use (\ref{fonccarrandomintegralsmesurecomplexeisotrope}) to obtain the same conclusion. This concludes the proof.
\epf

The third lemma replaces \cite[Proposition 2.4]{AGNEESSENS_2026} but requires geometric conditions on the Banach space $E$. Note however that these geometric assumptions naturally appear in the examples studied in \cite{AGNEESSENS_2026}. In the case of $\alpha$-stable random measures with $\alpha<2$, one cannot expect convergence in $L^p(\Omega;E)$ for all $1\leq p<\infty$, since $\alpha$-stable distributions do not possess finite moments of order $p\geq \alpha$.
\smallskip

\blm \label{lemmeconvergenceprobaintegrales}
Let $E$ be a real or complex separable Banach space of stable type $\alpha \in (0,2]$. Let $M$ be either a real symmetric $\alpha$-stable random measure on $(\R, \mathcal{B}_\R)$ if $E$ is real, or a complex isotropic $\alpha$-stable random measure on $(\R, \mathcal{B}_\R)$ if $E$ is complex, with the Lebesgue measure as control measure. Let $f : \R \to E$ be in $L^\alpha(\R;E)$. Then $\int_0^t f(s) M(ds)$ converges to $\int_0^\infty f(s) M(ds)$ in probability as $t \to \infty$.
\elm

\smallskip
\bpf
The proof follows from the inequality (\ref{inequationintegralesBanachtypestable}) and from the fact that $\int_t^\infty \lVert f(x) \rVert^\alpha m(dx)$ converges to $0$ as $t \to \infty$ since $f$ belongs to $L^\alpha(\R;E)$.
\epf

\smallskip
Finally, we will make use of the following general lemma for stable distributions in Banach spaces.

\smallskip

\blm \label{lemmasupportloistable}
Let $E$ be a real or complex separable Banach space and let $\mu$ be a probability measure on $E$. If $\mu$ is symmetric stable on $E$, then $\textrm{supp}(\mu)$ is a (closed) \textbf{real} subspace of $E$. Moreover, $\{ x^* \in E^* : \hat{\mu}(t x^*) = 1, \, \forall t \in \R \}^{\perp} \subseteq \textrm{supp}(\mu)$.
\elm

\smallskip
\bpf
The fact that $\textrm{supp}(\mu)$ is a real subspace of $E$ follows from Linde \cite[Section 6.9 and the proof of Prop. 6.9.2]{Linde86}. Let us prove the second part of Lemma \ref{lemmasupportloistable}.

Let $x \in E$ such that $x \notin \textrm{supp}(\mu)$. Then by the Hahn-Banach theorem there exists $x^* \in E_{\R}^*$ (the real dual of $E$) such that $x^* = 0$ on $\textrm{supp}(\mu)$ and $\lan x^*, x \ran \ne 0$. Moreover, there exists $a^* \in E^*$ such that $x^* = \Re(\lan a^*, \cdot \ran)$. In particular $\Re(\lan a^*, y \ran) = 0$ for $\mu$-almost every $y \in E$ and thus $\hat{\mu}(t a^*) = \int_E e^{i t \Re(\lan a^*, z \ran)} \mu(dz) = 1$ for every $t \in \R$. Since $\lan a^*, x \ran \ne 0$, we have $x \notin \{ y^* \in E^* : \hat{\mu}(t y^*) = 1, \forall t \in \R \}^{\perp}$, proving the second part of the lemma. 
\epf

\smallskip

We are now ready to prove Theorem \ref{thm1espacereeloucomplexemesurealéaréelleoucomplexe} by adapting the proof given in \cite[Prop 3.2]{AGNEESSENS_2026}. 

\smallskip

\bpf[Proof of Theorem \ref{thm1espacereeloucomplexemesurealéaréelleoucomplexe}]
Let $\mu$ be the distribution of the random vector $v$. 
Let us notice that $\mu$ is $T_s$-invariant for every $s \geq 0$. Indeed, let $A$ be a Borel subset of $E$. By the definitions of $\mu$ and $v$ and by Lemma \ref{lemmeloiintegralesintervallestranslatés} and Proposition \ref{propweaklystochintrealrandommeasure}, we have
\[
\mu(T_s^{-1}(A)) = \mathbb{P}(T_s(v) \in A) = \mathbb{P}\left(\int_{\mathbb{R}} u_{t-s} \, M(dt) \in A\right) = \mathbb{P}\left(\int_{\mathbb{R}} u_t \, M(dt) \in A\right),
\]
showing that $\mu(T_s^{-1}(A)) = \mathbb{P}(v \in A) = \mu(A)$. The measure $\mu$ is thus $T_s$-invariant.

\smallskip

We now show that $\mu$ is $(T_t)_{t \geq 0}$-strongly mixing. Let $f$ and $g$ be two bounded and uniformly continuous complex-valued functions defined on $E$. We aim to show that
\[
\lim_{t \to \infty} \int_E (f \circ T_t)g \, d\mu = \int_E f \, d\mu \int_E g \, d\mu.
\]
Since the set of bounded uniformly continuous functions on $E$ is dense in $L^2(E;\C)$ by Coudène \cite[Thm. 18.1]{Coudene16}, this will imply the claim. First, by definition of $\mu$, this is equivalent to  
\begin{equation} \label{equationeqmelangefortthm1}
\lim_{t \to \infty} \int_{\Omega} f(T_t(v))g(v) \, d\mathbb{P} = \int_{\Omega} f(v) \, d\mathbb{P} \int_{\Omega} g(v) \, d\mathbb{P}. 
\end{equation}
Let $\varepsilon > 0$. Since $f$ and $g$ are uniformly continuous and bounded, by Lemma \ref{lemmeconvergenceprobaintegrales}, there exists $N \geq 0$ such that
\begin{equation}
\left\| g\left(\int_{- \infty}^N u_t \, M(dt)\right) - g(v) \right\|_{L^1(\Omega; \C)} < \varepsilon \label{ineqsurgpreuvethm1}
\end{equation}
and
\begin{equation}
\left\| f\left(\int_{-N}^\infty u_t \, M(dt)\right) - f(v) \right\|_{L^1(\Omega; \C)} < \varepsilon. \label{ineqsurfpreuvethm1}
\end{equation}
Indeed, the sequence $(g(\int_{-\infty}^N u_t M(dt)))_{N \geq 0}$ converges in probability to $g(v)$ since $g$ is uniformly continuous and this sequence in bounded in $L^p(\Omg;\C)$ for every $p \geq 1$. Thus the convergence holds in $L^1(\Omg;\C)$ by Fatou's inequality and by Hölder's inequality.
\smallskip

Let $s > 2N$ and write
\begin{align}
\notag f(T_s v)g(v) &= f(T_s v)g(v) - f(T_s v)g\left(\int_{-\infty}^N u_t \, M(dt)\right) \\ \notag
&\quad + f(T_s v)g\left(\int_{-\infty}^N u_t \, M(dt)\right) - f\left(\int_{-N+s}^\infty u_{t-s} \, M(dt)\right)g\left(\int_{-\infty}^N u_t \, M(dt)\right) \\
&\quad + f\left(\int_{-N+s}^\infty u_{t-s} \, M(dt)\right)g\left(\int_{-\infty}^N u_t \, M(dt)\right). \label{egalitesurfgpreuvethm1}
\end{align}
For the first two terms, we use the fact that $f$ is bounded and the inequality (\ref{ineqsurgpreuvethm1}) to obtain
\begin{align*}
\left| \int_{\Omega} f(T_s v)g(v) \, d\mathbb{P} - \int_{\Omega} f(T_s v)g\left(\int_{-\infty}^N u_t \, M(dt)\right) d\mathbb{P} \right| \\
\leq \|f\|_\infty \left\| g\left(\int_{-\infty}^N u_t \, M(dt)\right) - g(v) \right\|_{L^1(\Omega;\C)} \leq \|f\|_\infty \varepsilon.
\end{align*}
For the third and fourth terms, let us define the random vectors
\begin{alignat*}{2}
X &:= \int_{-\infty}^{-N+s} u_{t-s} \, M(dt), \quad & Y &:= \int_{-N+s}^\infty u_{t-s} \, M(dt), \\
U &:= \int_{-\infty}^{-N} u_t \, M(dt), \quad & V &:= \int_{-N}^\infty u_t \, M(dt).
\end{alignat*}
Then, Lemma \ref{lemmeloiintegralesintervallestranslatés} implies that $X$ and $U$ have the same distribution, and so are $Y$ and $V$. By Lemma \ref{lemmeindependanceintegrales}, $X$ and $Y$ are independent, and so are $U$ and $V$. Thus, the random variables $f(X+Y) - f(Y)$ and $f(U+V) - f(V)$ are identically distributed and we obtain that 
\begin{align*}
&\left| \int_{\Omega} \left( f(T_s v )g\left(\int_{-\infty}^N u_t \, M(dt)\right) - f\left(\int_{-N+s}^\infty u_{t-s} \, M(dt)\right)g\left(\int_{-\infty}^N u_t \, M(dt)\right) \right) d\mathbb{P} \right| \\
&\leq \|g\|_\infty \left\| f\left(\int_{-\infty}^\infty u_{t-s} \, M(dt)\right) - f\left(\int_{-N+s}^\infty u_{t-s} \, M(dt)\right) \right\|_{L^1(\Omega; \C)} \\
&= \|g\|_\infty \left\| f\left(\int_{-\infty}^\infty u_t \, M(dt)\right) - f\left(\int_{-N}^\infty u_t \, M(dt)\right) \right\|_{L^1(\Omega; \C)} \\
&\leq \|g\|_\infty \varepsilon,
\end{align*}
where we used (\ref{ineqsurfpreuvethm1}) for the last inequality. For the last term of (\ref{egalitesurfgpreuvethm1}), since $s > 2N$, and by using Lemma \ref{lemmeindependanceintegrales}, one gets
\begin{align*}
&\int_{\Omega} f\left(\int_{-N+s}^\infty u_{t-s} \, M(dt)\right) g\left(\int_{-\infty}^N u_t \, M(dt)\right) d\mathbb{P} \\
&= \int_{\Omega} f\left(\int_{-N+s}^\infty u_{t-s} \, M(dt)\right) d\mathbb{P} \int_{\Omega} g\left(\int_{-\infty}^N u_t \, M(dt)\right) d\mathbb{P} \\
&= \int_{\Omega} f\left(\int_{-N}^\infty u_t \, M(dt)\right) d\mathbb{P} \int_{\Omega} g\left(\int_{-\infty}^N u_t \, M(dt)\right) d\mathbb{P}.
\end{align*}
Therefore, using again (\ref{ineqsurgpreuvethm1}) and (\ref{ineqsurfpreuvethm1}), we obtain
\begin{align*}
&\left| \int_{\Omega} f\left(\int_{-N+s}^\infty u_{t-s} \, M(dt)\right) g\left(\int_{-\infty}^N u_t \, M(dt)\right) d\mathbb{P} - \int_{\Omega} f(v) \, d\mathbb{P} \int_{\Omega} g(v) \, d\mathbb{P} \right| \\
&\quad \leq \|f\|_\infty \left\| g\left(\int_{-\infty}^N u_t \, M(dt)\right) - g(v) \right\|_{L^1(\Omega;\C)} + \|g\|_\infty \left\| f\left(\int_{-N}^\infty u_t \, M(dt)\right) - f(v) \right\|_{L^1(\Omega;\C)} \\
&\quad \leq \|f\|_\infty \varepsilon + \|g\|_\infty \varepsilon.
\end{align*}
We conclude that
\[
\left| \int_{\Omega} f(T_s(v))g(v) \, d\mathbb{P} - \int_{\Omega} f(v) \, d\mathbb{P} \int_{\Omega} g(v) \, d\mathbb{P} \right| \leq 2\|f\|_\infty \varepsilon + 2\|g\|_\infty \varepsilon,
\]
proving that (\ref{equationeqmelangefortthm1}) holds and that the measure $\mu$ is thus $(T_t)_{t \geq 0}$-strongly mixing.

\smallskip

Let us now show that the measure $\mu$ has full support. By Lemma \ref{lemmasupportloistable}, we have $\{ x^* \in E^* : \hat{\mu}(t x^*) = 1 , \forall t \in \R \}^{\perp} \subseteq \textrm{supp}(\mu)$. But if $x^* \in E^*$ is such that $\hat{\mu}(x^*) = 1$ then $\int_{\R} \lvert \lan x^*, u_t \ran\rvert^\alpha m(dt) = 0$ by (\ref{fonccaractrandomintBanachMreelle}) and (\ref{fonccarrandomintegralsmesurecomplexeisotrope}), and hence $\lan x^*, u_t \ran = 0$ for every $t \in \R \setminus N_{x^*}$ for some Borel subset $N_{x^*}$ of $\R$ with $m(N_{x^*}) = 0$. Thus $x^*= 0$ by assumption and the measure $\mu$ has full support.

\smallskip
This concludes the proof of Theorem \ref{thm1espacereeloucomplexemesurealéaréelleoucomplexe}.
\epf

We now give another proof of the mixing property in Theorem \ref{thm1espacereeloucomplexemesurealéaréelleoucomplexe} for a complex Banach space when $0 < \alpha < 2 $. This new proof is based on \cite[Section 4]{MauPrivault} but adapted to continuous actions. 

\smallskip

\bpr \label{Criteremixingloisstablessemigpsop}
Let $E$ be a complex Banach space of stable type $\alpha \in (0,2)$, and let $\mu$ be a symmetric $\alpha$-stable distribution with spectral measure $\xi$. Let $(T_t)_{t \geq 0}$ be a $\mathcal{C}_0$-semigroup of operators on $E$ leaving $\mu$ invariant. Then $\mu$ is strongly mixing with respect to $\mu$ if and only if for every $x^* \in E^*$, we have
\begin{align*}
    \lim_{t \to \infty} \int_E \lvert \Re(\lan x^*, z \ran) \rvert^{\alpha/2} \lvert \Re(\lan T_t^* x^*, z \ran) \rvert^{\alpha/2} \xi(dz) = 0
\end{align*}
and
\begin{align*}
    \lim_{t \to \infty} \int_E \lvert \Re(\lan x^*, z \ran) \rvert^{\alpha/2} \lvert \Im(\lan T_t^* x^*, z \ran) \rvert^{\alpha/2} \xi(dz) = 0.
\end{align*}
\epr

\smallskip

\bpf
The proof of Proposition \ref{Criteremixingloisstablessemigpsop} is analogous to the proof of \cite[Cor. 4.1]{MauPrivault}. First, let us remark that $\mu$ is strongly mixing with respect to $(T_t)_{t \geq 0}$ if and only if for each $x^* \in E^*$, the $\R^2$-valued process $X_t^{x^*} := (\Re(\lan T_t^* x^*, X \ran), \Im(\lan T_t^* x^*, X \ran)), t \geq 0$, is strongly mixing, where $X $ has distribution $\mu$. Indeed, the proof of \cite[Lemma 2.1]{MauPrivault} carries over to semigroups by replacing discrete time by continuous time and $T^n$ ($n \in \mathbb{Z}_+$) by $T_t$ ($t \geq 0$). Moreover, the proofs in \cite[Lemma 3.1 and Thm. 3.2]{MauPrivault} relying on \cite[Cor. 2.5 and Thm. 3.2]{FuchsStelzer13} also work in continuous time. This concludes the proof of Proposition \ref{Criteremixingloisstablessemigpsop}.
\epf

Note that Proposition \ref{Criteremixingloisstablessemigpsop} and \cite[Cor. 4.1]{MauPrivault} also hold for $\alpha = 1$ for symmetric stable distributions, since the spectral measure is symmetric and concentrated on the unit sphere. Indeed, for every $s \in \R$ and $\alpha \in (0,2)$, we have 
\begin{align*}
    \lvert s \rvert^\alpha = -c_\alpha^{-1} \int_{0}^\infty (\cos(ts) -1) \, t^{-(1+\alpha)} dt,
\end{align*}
and thus for every $x^* \in E^*$, we have
\begin{align*}
    -\int_E \lvert \Re(\lan x^*, z \ran) \rvert^\alpha d\xi(z) &= c_\alpha^{-1} \int_E \int_0^{\infty} (\cos(t \Re(\lan x^*,z \ran)) - 1) \, t^{-(1+\alpha)} dt \\
    &= c_\alpha^{-1} \int_E \int_0^\infty (e^{i t \Re(\lan x^*, z \ran)} - 1 - i t \Re(\lan x^*, z \ran)k(t)) \, t^{-(\alpha+1)} dt \, d\xi(z) \\
    &= \tfrac{1}{2} c_\alpha^{-1} \int_E \int_{-\infty}^\infty (e^{i t \Re(\lan x^*, z \ran)} - 1 - i t \Re(\lan x^*, z \ran)k(t)) \, \lvert t \rvert^{-(\alpha+1)} dt \, d\xi(z)
\end{align*}
with $k(t) = \indic_{\{\lvert t \rvert <1\}}$ and $c_\alpha$ is a positive constant (see also Linde \cite[p. 103]{Linde86}). Here, we used that $\xi$ is symmetric and concentrated on the unit sphere of $E$.

\smallskip

We thus obtain a new proof for the mixing condition in Theorem \ref{thm1espacereeloucomplexemesurealéaréelleoucomplexe}.

\smallskip

\bpf[New proof of the mixing condition in Theorem \ref{thm1espacereeloucomplexemesurealéaréelleoucomplexe}]
We start by the case where $\alpha = 2$, i.e, where $\mu $ is Gaussian. Let us denote by $R$ the covariance operator of $\mu$. Then for every $x^*, y^* \in E^*$, by (\ref{propisometryintcomplexemesurecomplexe}),
\begin{align*}
    \lan y^*, R T_s^* x^* \ran &= \int_E \lan y^*, z \ran \overline{\lan T_s^* x^*, z \ran} \mu(dz) \\
    &= \int_\Omg \int_{\R} \lan y^*, u_t \ran M(dt) \cdot \overline{\int_{\R} \lan x^*, T_s u_t \ran M(dt)} d\P \\
    &= \E \left ( \int_{\R} \lan y^*, u_t \ran M(dt) \cdot \overline{\int_{\R} \lan x^*, T_s u_t \ran M(dt)} \right) \\
    &= 4 c_0 \int_{\R} \lan y^*, u_t \ran \overline{\lan x^*, u_{t-s} \ran} m(dt)
\end{align*}
which converges to $0$ as $s \to \infty$ as the convolution of two square-integrable functions. 

\smallskip
We now consider the case where $\alpha \in (0,2)$ and make use of the spectral measure (\ref{mesurespectralemesuremustablecomplexe}) of $\mu$. Let us notice that for every $x^* \in E^*$, the quantity
\begin{align*}
    &\int_E \lvert \Re(\lan x^*, z \ran) \rvert^{\alpha/2} \lvert \Re(\lan T_s^* x^*, z \ran) \rvert^{\alpha/2} \xi(dz) \\ = &\int_{\R \times \mathbb{S}_2} \lvert \Re(\lan x^*, u \, u_t \ran) \rvert^{\alpha/2} \lvert \Re(\lan x^*, u \, u_{t-s} \ran) \rvert^{\alpha/2} m(dt) \gamma(du)
\end{align*}
converges to $0$ as $s \to \infty$ as the convolution of two square-integrable functions, as well as the quantity
\begin{align*}
    \int_E \lvert \Re(\lan x^*, z \ran) \rvert^{\alpha/2} \lvert \Im(\lan T_s^* x^*, z \ran) \rvert^{\alpha/2} \xi(dz).
\end{align*}
The conclusion of the proof follows from Proposition \ref{Criteremixingloisstablessemigpsop}.
\epf
\smallskip
Note that it is also possible to give another proof of the mixing condition for real Banach spaces by adapting Proposition \ref{Criteremixingloisstablessemigpsop} to the real setting and by using the spectral representation (\ref{spectralrepresentationrealBanach}).

\medskip

We now move to the second main result of this paper for semigroups, corresponding to the analogue of \cite[Thm. 4.1]{AGNEESSENS_2026}. 

\smallskip

\bth \label{thm2espacecomplexemesurealéacomplexe}
Let $E$ be a complex infinite dimensional separable Banach space of stable type $\alpha \in (0,2]$ and let $J$ be an interval. Let $(T_s)_{s \geq 0}$ be a $\mathcal{C}_0$-semigroup of operators on $E$. Suppose that there exists a family $(u_t)_{t \in J}$ in $E$ with $t \mapsto u_t \in L^\alpha(J;E)$ such that $T_s u_t = e^{ist} u_{t}$ for every $t \in J$ and $s \geq 0$. Assume moreover that $\textrm{span}\{ u_t : t \in J \setminus N \}$ is dense in $E$ for every Borel subset $N \subset J$ with $m(N) = 0$.
Then, for every complex isotropic $\alpha$-stable random measure $M$ on $(J,\mathcal{B}_J)$, the distribution of the $E$-valued random vector
$$
v:= \int_{J} u_t \, M(dt)
$$
is invariant under $(T_s)_{s \geq 0}$ with full support and is isotropic $\alpha$-stable. Such a distribution is strongly mixing with respect to $(T_s)_{s \geq 0}$ if $\alpha = 2$, and never ergodic with respect to $(T_s)_{s \geq 0}$ if $\alpha \in (0,2)$.
\eth

\smallskip

In order to prove Theorem \ref{thm2espacecomplexemesurealéacomplexe}, we need the following elementary lemma.

\smallskip

\blm \label{lemmeintégralefoncperiodique}
Let $f : [0,\infty) \to \R$ be a $d$-periodic function in $L^1([0,d];\R)$. Then 
\begin{align*}
    \frac{1}{S} \int_0^S f(x) m(dx) \underset{S \to \infty}{{\longrightarrow}} \, \frac{1}{d} \int_0^d f(x) m(dx).
\end{align*}
\elm

\bpf
Let $S = d  N(S) + r(S) $ with $N(S) \in \Z_+$ and $0 \leq r(S) < d$. Then
\begin{align*}
    \int_0^S f(x) m(dx) &= \sum_{j=0}^{N(S) -1} \int_{d j}^{d(j+1)} f(x) m(dx) + \int_{d N(S)}^{d N(S) + r(S)} f(x) m(dx) \\
    &= N(S) \int_0^d f(x) m(dx) + \int_0^{r(S)} f(x) m(dx).
\end{align*}
Since
\[
\left|\int_0^{r(S)} f(x)\,m(dx)\right|
\le \int_0^d |f(x)|\,m(dx)
\quad \textrm{and} \quad
\frac{N(S)}{S} \to \frac{1}{d},
\]
we obtain
\[
\frac{1}{S}\int_0^S f(x)\,m(dx)
\to
\frac{1}{d}\int_0^d f(x)\,m(dx)
\]
as $S \to \infty$,
concluding the proof of Lemma \ref{lemmeintégralefoncperiodique}.
\epf

\smallskip

\bpf[Proof of Theorem \ref{thm2espacecomplexemesurealéacomplexe}]
Let us denote by $\mu$ the distribution of the random vector $v$. We first show that $\mu$ is invariant with respect to $(T_s)_{s \geq 0}$. For every $x^* \in E^*$ and every $s \geq 0$,
\begin{align*}
    \hat{\mu}(T_s^* x^*) &= \exp(-c_0 \int_{J} \lvert \lan T_s^* x^*, u_t \ran \rvert^\alpha m(dt)) \\
    &= \exp(-c_0 \int_{J} \lvert \lan x^*, T_s u_t \ran \rvert^\alpha m(dt)) \\
    &= \exp(-c_0 \int_{J} \lvert \lan  x^*, e^{ist} u_t \ran \rvert^\alpha m(dt)) \\
    &= \exp(-c_0\int_{J} \lvert \lan x^*, u_t \ran \rvert^\alpha m(dt)) \\
    &= \hat{{\mu}}(x^*)
\end{align*}
showing that $\mu$ and $\mu \circ T_s^{-1}$ have the same characteristic function and hence that $\mu$ is invariant with respect to $(T_s)_{s \geq 0}$. 

\smallskip
As in the proof of Theorem \ref{thm1espacereeloucomplexemesurealéaréelleoucomplexe}, we can show that the measure $\mu$ has full support using Lemma \ref{lemmasupportloistable} and the assumption on the family $(u_t)_{t \in J}$.

\smallskip
Suppose now that $\alpha = 2$ and let us prove that $\mu$ is strongly mixing. Let $R$ be the covariance operator of the Gaussian random vector $v$. Using (\ref{propisometryintcomplexemesurecomplexe}), we have for every $x^*, y^* \in E^*$ and for every $s \geq 0$,
\begin{align*}
    \lan y^*, R T_s^* x^* \ran &= \int_E \lan y^*,z \ran \overline{\lan T_s^* x^*,z \ran} \, \mu(dz) \\
    &=\int_\Omg \int_{J} \lan y^*, u_t \ran M(dt) \cdot \overline{\int_{J} \lan x^*, T_s u_t \ran M(dt)} \, d\P \\
    &= \E \left(\int_{J} \lan y^*, u_t \ran M(dt) \cdot \overline{\int_{J} \lan x^*, T_s u_t \ran M(dt)} \right) \\
    &= 4 c_0 \int_{J} e^{-ist} \lan y^*, u_t \ran \, \overline{\lan x^*, u_t \ran} \,m(dt) 
\end{align*}
which converges to zero as $s \to \infty$ since $t \mapsto u_t \in L^2(J;E)$. This shows that $\mu$ is strongly mixing when $\alpha = 2$.

\smallskip
It remains to prove that the measure $\mu$ is not ergodic with respect to the semigroup when $\alpha \in (0,2)$. By \cite[p. 19]{CornfeldFominSinai82}, it is enough to show that, for some $x^* \in E^*$, the quantity
\begin{align*}
    \frac{1}{S} \int_0^S \E_{\mu}(e^{i \Re(\lan T_s^* x^* - x^*, \cdot \ran)}) m(ds)
\end{align*}
does not converge to $\int_E e^{i \Re(\lan x^*, z \ran)} \mu(dz) \cdot \int_E e^{-i \Re(\lan x^*, z \ran)} \mu(dz)$. By (\ref{fonccarrandomintegralsmesurecomplexeisotrope}), it is enough to show that 
\begin{align*}
    \frac{1}{S} \int_0^S \exp \left(- c_0 \int_{J} \lvert 1 - e^{ist} \rvert^\alpha \lvert \lan x^*, u_t \ran \rvert^\alpha m(dt) + 2 c_0 \int_{J} \lvert \lan x^*, u_t \ran \rvert^\alpha m(dt) \right) m(ds) 
\end{align*}
does not converge to $1$. But by Jensen's inequality we have
\begin{align*}
    &\frac{1}{S} \int_0^S \exp \left(- c_0 \int_{J} \lvert 1 - e^{ist} \rvert^\alpha \lvert \lan x^*, u_t \ran \rvert^\alpha m(dt) + 2 c_0 \int_{J} \lvert \lan x^*, u_t \ran \rvert^\alpha m(dt) \right) m(ds) \\
    &\geq \exp \left(\frac{1}{S} \int_0^S \left(- c_0 \int_{J} \lvert 1 - e^{ist} \rvert^\alpha \lvert \lan x^*, u_t \ran \rvert^\alpha m(dt) + 2 c_0 \int_{J} \lvert \lan x^*, u_t \ran \rvert^\alpha m(dt) \right)  m(ds) \right).
\end{align*}

Now let us notice that
\begin{align*}
    &\frac{1}{S} \int_0^S \left(2 c_0 \int_{J} \lvert \lan x^*, u_t \ran \rvert^\alpha m(dt) - c_0 \int_{J} \lvert 1 - e^{ist} \rvert^\alpha \lvert \lan x^*, u_t \ran \rvert^\alpha m(dt) \right) m(ds) \\
    &= 2 c_0 \int_{J} \lvert \lan x^*, u_t \ran \rvert^\alpha m(dt) - 2^\alpha c_0 \int_{J} \left( \frac{1}{S}  \int_0^S \lvert \sin(\frac{st}{2}) \rvert^\alpha m(ds) \right) \lvert \lan x^*, u_t \ran \rvert^\alpha m(dt) \\
    &= 2 c_0 \int_{J} \lvert \lan x^*, u_t \ran \rvert^\alpha m(dt) - 2^\alpha c_0 \int_{J} \left( \frac{2}{S \lvert t \rvert}  \int_0^{\frac{S \lvert t \rvert}{2}} \lvert \sin(v) \rvert^\alpha m(dv) \right) \lvert \lan x^*, u_t \ran \rvert^\alpha m(dt)  
\end{align*}
which converges as $S \to \infty$ to
\begin{align*}
    &2 c_0 \int_{J} \lvert \lan x^*, u_t \ran \rvert^\alpha m(dt) - 2^\alpha c_0 \int_{J} \lvert \lan x^*, u_t \ran \rvert^\alpha \left( \frac{1}{2\pi} \int_{0}^{2 \pi} \lvert \sin(v) \rvert^\alpha m(dv) \right) m(dt) \\
   & = 2 c_0 (1-2^{\alpha -1} c_0) \int_{J} \lvert \lan x^*, u_t \ran \rvert^\alpha m(dt)
\end{align*}
by Lemma \ref{lemmeintégralefoncperiodique} and by Lebesgue's convergence theorem. Since $\int_{J} \lvert \lan x^*, u_t \ran \rvert^\alpha m(dt) >0$ for every $x^* \ne 0$, it is enough to show that $1- 2^{\alpha -1} c_0 >0$ to conclude that $\mu$ is not ergodic with respect to $(T_s)_{s \geq 0}$. By Jensen's inequality again, we have
\begin{align*}
    \frac{1}{2\pi} \int_0^{2\pi} \lvert \cos(x) \rvert^\alpha m(dx) \leq \left(\frac{1}{2\pi} \int_0^{2\pi} (\cos(x))^2 m(dx) \right)^{\alpha/2},
\end{align*}
that is, $c_0 \leq 2^{-\alpha/2}$. Thus $2^{\alpha -1} c_0 \leq 2^{\alpha - 1 - \alpha/2 } = 2^{\alpha/2 -1} <1$ since $\alpha \in (0,2)$. This concludes the proof of Theorem \ref{thm2espacecomplexemesurealéacomplexe}.
\epf

\subsection{Applications}

We illustrate our main results for semigroups with different examples. 

\smallskip
Let us consider chaotic translation semigroups within the framework of weighted spaces. Let $I \in \{[0, \infty[, \mathbb{R}\}$. For a fixed $1 \leq p < \infty$ and a measurable weight function $\rho : I \longrightarrow (0, \infty)$, define the following complex Banach space:
\[
L^p_\rho(I) := \left\{ f : I \longrightarrow \mathbb{C} : \int_I |f(x)|^p \rho(x) \, dx < \infty \right\}
\]
equipped with the norm $\|f\| := \left( \int_I |f(x)|^p \rho(x) \, dx \right)^{1/p}$. A weight function $\rho$ is called \textit{admissible} if there exist constants $M \geq 1$ and $w \in \mathbb{R}$ such that $\rho(s) \leq M e^{wt} \rho(t+s)$ for every $s \in I$ and $t \geq 0$.
For such an admissible weight, the translation semigroup $(T_t)_{t \geq 0}$ on $L^p_\rho(I)$ is given by:
\begin{align}
    T_t(f)(x) := f(x+t), \quad x \in I,\, t \geq 0.
\end{align}

\smallskip

\bth \label{thmexempletranslationsLpstable}
Let $1 \leq p < \infty$, $I \in \{ [0,\infty), \mathbb{R} \}$, and let $\rho : I \to (0,\infty)$ be an integrable admissible weight such that the moment condition
\[
\int_I |x|^{k p} \rho(x) \, dx < \infty
\]
is satisfied for some integer $k \geq 1$. For any positive, smooth, and integrable function $\varphi: \mathbb{R} \to \mathbb{R}$, we define the map $F: \mathbb{R} \to L^p_\rho(I)$ by
\[
F(\theta)(x) = \varphi(\theta)e^{i\theta x} \quad \text{for all } \theta \in \mathbb{R} \text{ and } x \in I.
\]
Then, the vector 
\[
u_t := \int_{\mathbb{R}} e^{-it\theta} F(\theta) \, d\theta
\]
is well-defined for all $t \in \mathbb{R}$. Furthermore, let $v=\int_{\mathbb R} u_t\,M(dt)$, where $M$ is a complex isotropic $\alpha$-stable random measure on $(\mathbb R,\mathcal B_{\mathbb R})$ with the Lebesgue measure as control measure. Suppose that $\varphi^{(j)}$ is integrable for every $1 \leq j \leq k$. Then $v$ is almost surely frequently hypercyclic for the translation semigroup $(T_t)_{t\geq 0}$ whenever $1/k<\alpha<p$ if $p<2$, and whenever $1/k<\alpha\leq 2$ if $p\geq 2$.
\eth

\smallskip
\begin{proof}
We follow the same idea as \cite[Thm. 4.2]{AGNEESSENS_2026}.
Let $\varphi : \mathbb{R} \longrightarrow \mathbb{R}$ be a positive, smooth, and integrable function on $\mathbb{R}$.  
The function $F$ is thus Bochner-integrable, ensuring that the vector
\[
u_t := \int_{\mathbb{R}} e^{-it\theta} F(\theta) \, d\theta
\]
is well defined for all $t \in \mathbb{R}$. Let us observe that $T_s F(\theta) = e^{is\theta} F(\theta)$ for all $\theta \in \mathbb{R}$ and $s \geq 0$, implying that $T_s u_t = u_{t-s}$ for all $s \geq 0$ and $t \in \mathbb{R}$. Moreover, since $\int_I \lvert x \rvert^{k p} \rho(x) dx < \infty $ and since $\varphi^{(j)}$ is integrable for every $j \leq k$, the function $F$ is $k$-times differentiable and satisfies $\|u_t\| \leq C/\lvert t\rvert^k$ for every $\lvert t \rvert \geq 1$, for some constant $C > 0$ (see \cite[Proof of Lemma 2.4.3]{HytonenvanNeerverVeraarWeis16}). Thus, since $\alpha > 1/k$, we have that $t \mapsto u_t \in L^\alpha(\R;E)$.

\smallskip
We now show that $\textrm{span}\{ u_t : t \in \R \setminus N \}$ is dense in $L^p_\rho(I)$ for every Borel subset $N \subset \R$ with $m(N) = 0$. Let $\phi^* \in (L^p_\rho(I))^*$ such that $\lan \phi^*, u_t \ran = 0$ for every $t \in \R \setminus N$. Then 
$$
\lan \phi^*, u_t \ran = \int_{\R} e^{-i t \theta} \lan \phi^*, F(\theta) \ran \, d\theta = 0 \quad \textrm{for every $t \in \R \setminus N$}
$$
and since the map $\theta \mapsto \lan \phi^*, F(\theta) \ran$ is integrable we deduce that $\lan \phi^*, F(\theta) \ran = 0$ almost everywhere. Moreover, since $\phi^* \in (L^p_\rho(I))^* = L^{p'}_{\rho^{-p'/p}}(I) $ with $p'$ the conjugate exponent of $p$, the functional $\phi^*$ can be represented by $\phi \in L^{p'}_{\rho^{-p'/p}}(I)$ and we obtain 
$$
\int_I \phi(x) e^{i \theta x} dx = 0 \quad \textrm{for a.e $\theta \in \R$},
$$
implying that $\phi = 0$ and thus that $\phi^* = 0$, since $\rho$ is integrable. This shows the density of $\textrm{span}\{ u_t : t \in \R \setminus N \}$ in $L^p_\rho(I)$.

\smallskip
By Theorem \ref{thm1espacereeloucomplexemesurealéaréelleoucomplexe}, the proof of Theorem \ref{thmexempletranslationsLpstable} is complete since $L_\rho^p(I)$ is of stable type $2$ if $p \geq 2$, and of stable type $\alpha<p$ if $p<2$.
\end{proof}

\smallskip

\noindent For a bounded sequence of nonzero complex numbers $(w_n)_{n \geq 1}$, we recall that the backward weighted shift operator $A \colon \ell_p(\mathbb{Z}_+) \to \ell_p(\mathbb{Z}_+)$ is defined on the canonical basis $(e_n)_{n \geq 0}$ of $\ell_p(\mathbb{Z}_+)$ by $A e_0 = 0$ and $A e_n = w_n e_{n-1}$ for every $n \geq 1$.

\smallskip

\begin{theorem} \label{theoremexemplesemigpshiftspoids}
Let $1 \leq p < \infty$ and $A$ be the weighted shift on the complex space $E = \ell_p(\mathbb{Z}_+)$ with sequence of weights $(w_n)_{n \geq 1}$. Set $\beta_n := w_1 \cdots w_n$ for all $n \geq 1$ and $\beta_0 := 1$. If $\rho := \displaystyle \limsup_{n \to \infty} (1/|\beta_n|)^{1/n}$ is finite, then for every $\eta \in (0,1/\rho)$, for every $0 <\alpha < \min(p,2)$ and for every complex isotropic $\alpha$-stable random measure $M$ on $[-\eta, \eta]$, the distribution of the random vector
\[
v := \int_{-\eta}^{\eta} \left( \frac{i^n t^n}{\beta_n} \right)_{n \geq 0} M(dt)
\]
is invariant with respect to $(T_t)_{t \geq 0} := (e^{tA})_{t \geq 0}$ with full support and is isotropic $\alpha$-stable. Moreover, if $p=2$, then the distribution of $v$ is an invariant Gaussian isotropic strongly mixing measure with full support with respect to $(T_t)_{t \geq 0}$ for every complex isotropic $2$-stable random measure on $[- \eta, \eta]$. 
\end{theorem}

\smallskip

\begin{proof}
Following \cite[Thm. 4.4]{AGNEESSENS_2026}, let $v_{\lambda} := (\lambda^n / \beta_n)_{n \geq 0}$ for every $\lambda \in \C$. Then $v_{\lambda} \in \ell_p(\mathbb{Z}_+)$ if and only if $\displaystyle \sum_{n \geq 0} |\lambda|^{np} / |\beta_n|^p < \infty$, and in that case, we have $A(v_{\lambda}) = (\lambda^{n+1} / \beta_n)_{n \geq 0} = \lambda v_{\lambda}$. Note that the series $\displaystyle \sum_{n \geq 0} \lambda^n / \beta_n$ has a radius of convergence $R := 1/\rho$.
Let $\eta \in (0, R)$ and let $u_s := v_{is}$ for every $-\eta < s < \eta$. Then $u_s \in \ell_p(\mathbb{Z}_+)$ and $T_t u_s = e^{i s t} u_s$ for every $-\eta < s < \eta$ and $t \geq 0$, since $A$ is the generator of the semigroup $(T_t)_{t \geq 0}$. 

\smallskip
Now fix $0< \alpha <p$ and let $M$ be a complex isotropic $\alpha$-stable random measure on $[-\eta, \eta]$.
Let us notice that the map $s \mapsto u_s$ belongs to $L^2([-\eta, \eta];E)$ and thus to $L^\alpha([-\eta, \eta]; E)$. Indeed, this easily follows from the inequalities
\begin{align*}
    \int_{-\eta}^{\eta} \lVert u_t \rVert^2 dt &\leq 2 \eta + \int_{- \eta}^{\eta} \left( \sum_{j \geq 0} \lvert \lan e_j^*, u_t \ran \rvert^p \right)^{2/p} dt \\
    &\leq 2 \eta + 2 \eta \, \left(\sum_{j \geq 0} \frac{\eta^{pj}}{\lvert \beta_j \rvert^p} \right)^{2/p} \\
    &< \infty
\end{align*}
since $0 < \eta < R$. Moreover, it was shown in \cite[Thm. 4.4]{AGNEESSENS_2026} that $\textrm{span}\{ u_t : t \notin B \}$ is dense in $E$ for every Borel set $B$ of $[- \eta, \eta]$ such that $m(B) = 0$. Since $E$ is of stable type $\alpha$, we conclude that the distribution of $\int_{- \eta}^\eta u_t M(dt)$ is invariant with full support with respect to $(T_t)_{t \geq 0}$ by Theorem \ref{thm2espacecomplexemesurealéacomplexe}. The conclusion for the case $p = 2$ follows from the fact that $E$ is of stable type $2$. This concludes the proof of Theorem \ref{theoremexemplesemigpshiftspoids}.
\end{proof}

\section{Stable invariant measures for operators}

\subsection{Main results}

We study the analogue of Theorems \ref{thm1espacereeloucomplexemesurealéaréelleoucomplexe} and \ref{thm2espacecomplexemesurealéacomplexe}  for linear operators. In this case, we replace integrals over $J \subseteq \R$ by sums over $J \subseteq \Z$, as in \cite{AGNEESSENS_2023}.

\smallskip

Let $M$ be either a real symmetric $\alpha$-stable random measure on $(\Z,\mathcal{P}(\Z))$ or a complex isotropic $\alpha$-stable random measure on $(\Z,\mathcal{P}(\Z))$ with the counting measure $m$ as control measure. Let $f$ be a real- or complex-valued measurable function on $\Z$ in $\ell^\alpha(\Z)$, depending on whether the random measure $M$ is real or complex. Then the random integral $\int_{\Z} f \, dM$ is a sum of independent symmetric $\alpha$-stable random variables. Indeed, if $f$ has finite support, then  
\begin{align*}
    \int_{\Z} f \, dM = \sum_{n \in \Z} f(n) M(\{ n \}),
\end{align*}
and this relation extends to all $\alpha$-summable functions on $\Z$ by density arguments. Let us also notice that Lemmas \ref{lemmeindependanceintegrales}, \ref{lemmeloiintegralesintervallestranslatés} and \ref{lemmeconvergenceprobaintegrales} can easily be adapted to the case of a real or complex stable random measure on $(\Z, \mathcal{P}(\Z),m)$, where $m$ is the counting measure. With these adapted lemmas, we can prove the following analogue of Theorem \ref{thm1espacereeloucomplexemesurealéaréelleoucomplexe} for operators.
\smallskip

\begin{theorem} \label{thm1operatorsmesurealeareelleoucomplexstable}
    Let $E$ be a real or complex separable Banach space of stable type $\alpha \in (0,2]$. Let $T$ be a bounded linear operator on $E$ and suppose that there exists a sequence $(u_n)_{n \in \Z}$ of vectors in $E$ with $\displaystyle \sum_{n \in \Z} \lVert u_n \rVert^\alpha < \infty$ such that $T u_n = u_{n-1}$ for every $n \in \Z$. Let $M$ be either a real symmetric $\alpha$-stable random measure on $(\Z,\mathcal{P}(\Z))$ if $E $ is real, or a complex isotropic $\alpha$-stable random measure on $(\Z,\mathcal{P}(\Z))$ if $E$ is complex, with the counting measure as control measure. Suppose that $\textrm{span}\{ u_n : n \in \Z \}$ is dense in $E$. Then the distribution of the random vector
    $$
    v:= \displaystyle \sum_{n \in \Z} u_n X_n,
    $$
    where $X_n := M(\{n \})$ for every $n \in \Z$, is invariant and strongly mixing with full support under $T$, and $v $ is a $S\alpha S$ random vector in $E$. 
\end{theorem}

\smallskip
Note that Theorem \ref{thm1operatorsmesurealeareelleoucomplexstable} corresponds to \cite[Prop. 2.5]{AGNEESSENS_2023} but in the context of stable random variables and is thus less general. However, the proof in \cite[Prop 2.5]{AGNEESSENS_2023} relies on \cite[Lemmas 2.2 and 2.3]{AGNEESSENS_2023} for general random variables, whereas our Theorem \ref{thm1operatorsmesurealeareelleoucomplexstable} rely on the properties of stable random integrals. Note that we could also give another proof of the mixing condition in the case where $E$ is real or complex and $\alpha \in (0,2]$ as we did for semigroups, using \cite[Cor. 4.1]{MauPrivault} as well as the covariance operator.

\smallskip

The analogue of Theorem \ref{thm2espacecomplexemesurealéacomplexe} for operators is the following.
\smallskip

\bth \label{thm2operatormesurealeacomplexstable}
Let $E$ be a complex separable Banach space of stable type $\alpha \in (0,2]$ and let $J $ be a subset of $\Z$. Let $T$ be a bounded linear operator on $E$ and suppose that there exists a sequence $(u_n)_{n \in J}$ of vectors in $E$ with $\displaystyle \sum_{n \in J} \lVert u_n \rVert^\alpha < \infty$ such that $T u_n = \lambda_n u_{n}$ for every $n \in J$ with $\lambda_n \in \T$. Let $M$ be a complex isotropic $\alpha$-stable random measure on $(J,\mathcal{P}(J))$ with the counting measure as control measure. Suppose that $\textrm{span}\{ u_n : n \in J \}$ is dense in $E$. Then the distribution of the random vector
    $$
    v:= \displaystyle \sum_{n \in J} u_n X_n,
    $$
    where $X_n := M(\{n \})$ for every $n \in J$, is invariant with full support under $T$, and $v $ is an isotropic $\alpha$-stable random vector. Such a distribution is never ergodic. 
\eth

\smallskip

\bpf
Using the characteristic function (\ref{fonccarrandomintegralsmesurecomplexeisotrope}), as well as Lemma \ref{lemmasupportloistable}, we can show as in the proof of Theorem \ref{thm2espacecomplexemesurealéacomplexe} that the distribution of the random vector $v$ is $T$-invariant with full support. It remains to show that the distribution of $v$ is not ergodic with respect to $T$ when $\alpha \in (0,2)$. Indeed, the case $\alpha = 2$ was treated by Grivaux \cite[Thm. 5.1]{Grivaux11}. Let us denote by $\mu$ the distribution of $v$.

\smallskip
Exactly as in the proof of Theorem \ref{thm2espacecomplexemesurealéacomplexe}, it is enough to show that
\begin{align*}
    \frac{1}{N} \sum_{k=1}^N \E_{\mu}(e^{i \Re(\lan T^{*k} x^* - x^*, \cdot \ran)}) 
\end{align*}
does not converge to $\E_\mu (e^{i \Re(\lan x^*, \cdot \ran)}) \cdot \E_\mu(e^{-i \Re(\lan x^*, \cdot \ran)})$ in order to prove that $\mu$ is not ergodic with respect to $T$, that is, it is enough to show that
\begin{align*}
    \frac{1}{N} \sum_{k=1}^N \exp \left(- c_0 \sum_{n \in J} \lvert 1 - \lambda_n^k \rvert^\alpha \lvert \lan x^*, u_n \ran \rvert^\alpha + 2 c_0 \sum_{n \in J} \lvert \lan x^*, u_n \ran \rvert^\alpha  \right)  
\end{align*}
does not converge to $1$ by (\ref{fonccarrandomintegralsmesurecomplexeisotrope}). And again by Jensen's inequality, we have
\begin{align*}
    &\frac{1}{N} \sum_{k=1}^N  \exp \left(-c_0 \sum_{n \in J} \lvert 1-\lambda_n^k \rvert^\alpha \lvert \lan x^*, u_n \ran \rvert^\alpha + 2 c_0 \sum_{n \in J} \lvert \lan x^*, u_n \ran \rvert^\alpha \right) \\
    &\geq \exp \left(2 c_0 \sum_{n \in J} \lvert \lan x^*, u_n \ran \rvert^\alpha - c_0 \sum_{n \in J}  \frac{1}{N} \sum_{k=1}^N \lvert 1 - \lambda_n^k \rvert^\alpha \lvert \lan x^*, u_n \ran \rvert^\alpha \right).  
\end{align*}
Let us denote by $a_N(\lambda) = \frac{1}{N} \displaystyle \sum_{k=1}^N \lvert 1 - \lambda^k \rvert^\alpha$ for $N \geq 1$ and $\lambda \in \T$. If $\lambda \in \T$ is not a root of unity, then (see for example \cite[Ex. 10.4]{EisnerFarkas25}) the sequence $(\lambda^n)_{n \geq 0}$ is uniformly distributed in the circle and thus $a_N(\lambda) \underset{N \to \infty}{{\longrightarrow}} \frac{1}{2 \pi} \int_0^{2 \pi} \lvert 1 - e^{i t} \rvert^\alpha dt = 2^\alpha c_0$. If $\lambda $ is a $q$-th root of unity for some $q \geq 1$, then $a_N(\lambda) \underset{N \to \infty}{{\longrightarrow}} \frac{1}{q} \displaystyle \sum_{k=0}^{ q - 1} \lvert 1 - \lambda^k \rvert^\alpha $. Let $a(\lambda) := \displaystyle \lim_{N \to \infty} a_N(\lambda)$ for $\lambda \in \T$.  Then 
\[
2 c_0 \sum_{n \in J} \lvert \lan x^*, u_n \ran \rvert^\alpha - c_0 \sum_{n \in J}  \frac{1}{N} \sum_{k=1}^N \lvert 1 - \lambda_n^k \rvert^\alpha \lvert \lan x^*, u_n \ran \rvert^\alpha
\]
converges as $N \to \infty$ to $c_0 \displaystyle \sum_{n \in J} (2 - a(\lambda_n)) \lvert \lan x^*, u_n \ran \rvert^\alpha$ by Lebesgue's convergence theorem. To conclude that $\mu$ is not ergodic with respect to $T$, it is enough to show that $$\displaystyle \sum_{n \in J} (2 - a(\lambda_n)) \lvert \lan x^*, u_n \ran \rvert^\alpha > 0$$ 
for some $x^* \in E^*$. 

\smallskip
Let us notice that if $\lambda^q = 1$ for some $q \geq 2$ and $\lambda \ne 1$, then
\begin{align*}
    \frac{1}{q } \displaystyle \sum_{j=0}^{q-1} \lvert 1 - \lambda^j \rvert^\alpha &\leq  \left( \frac{1}{q} \sum_{j=0}^{q-1} \lvert 1 - \lambda^j \rvert^2 \right)^{\alpha/2} \\ 
    &\leq 2^{\alpha/2}
\end{align*}
since $\lvert 1 - \lambda^j \rvert^2 = 2 - 2 \Re(\lambda^j)$ and $\lambda^q = 1$.
Since $2 - 2^\alpha c_0 >0$ when $\alpha \in (0,2)$ by the proof of Theorem \ref{thm2espacecomplexemesurealéacomplexe}, we obtain that $2 - a(\lambda_n) >0 $ for every $n \in J$ when $\alpha \in (0,2)$, and thus $\displaystyle \sum_{n \in J} (2 - a(\lambda_n)) \lvert \lan x^*, u_n \ran \rvert^\alpha > 0$ for every $x^* \ne 0$. This concludes the proof of Theorem \ref{thm2operatormesurealeacomplexstable}.
\epf

Note that Grivaux's result \cite[Thm. 5.1]{Grivaux11}, corresponding to the case $\alpha = 2$ in Theorem \ref{thm2operatormesurealeacomplexstable}, strongly relies on the square integrability of the random variables $X_n$, $n \in J$. Therefore, the case $\alpha \in (0,2)$ in Theorem \ref{thm2operatormesurealeacomplexstable} provides an example of a sequence $(X_n)_{n \in J}$ of non-square-integrable complex rotationally invariant random variables such that the distribution of the series
\[
\sum_{n \in J} u_n X_n
\]
is $T$-invariant with full support but not ergodic.

\subsection{Applications}

Let $E$ be either $\ell_p(\Z_+)$, $\ell_p(\Z)$, $c_0(\Z_+)$ or $c_0(\Z)$ with $1 \leq p < \infty$. Recall that the unilateral (resp.\ bilateral) weighted shift on $E$ with sequence of weights $(w_n)_{n \geq 1}$ (resp. $(w_n)_{n \in \Z}$) is defined by $B_w e_0 = 0$ and $ B_w e_n=w_n e_{n-1}$ for all $n\geq 1$ (resp.
$B_w (e_n)=w_n e_{n-1}$ for all  $n\in\mathbb Z$), where $(w_n)_n$ is a bounded sequence of nonzero scalars called the \emph{weight sequence}.
If $B_w$ is a unilateral weighted shift on $\ell_p(\Z_+)$ or $c_0(\Z_+)$, we set
\[
\beta_n := w_1 \cdots w_n \quad \text{for } n\geq 1
\quad \text{and} \quad
\beta_0 := 1.
\]
If $B_w$ is a bilateral weighted shift on $\ell_p(\Z)$ or $c_0(\Z)$, we set
\[
\beta_n := w_1 \cdots w_n \quad \text{for } n\geq 1,
\]
\[
\beta_n := \left(\prod_{k=n+1}^{0} w_k \right)^{-1}
\quad \text{if } n\leq -1,
\qquad \text{and} \qquad
\beta_0 := 1.
\]

\smallskip

Since $\ell_p$ is of stable-type $0 < \alpha <p$ if $p < 2$ and of stable-type $2$ if $p \geq 2$, we obtain the following result.

\smallskip

\begin{proposition}
    Let $1 \le p < \infty$ and let $B_w$ be a unilateral or bilateral weighted shift with weight sequence $(w_n)_n$ acting on $\ell_p(\mathbb Z_+)$ or $\ell_p(\mathbb Z)$ respectively. Let $0<\alpha<p$ if $p<2$ and $0<\alpha \leq2$ if $p \geq2$.
Assume that
\[
\sum_{n\ge1}\frac{1}{|w_1 \dotsm w_n|^\alpha}<\infty
\]
in the unilateral case, and that
\[
\sum_{n\in\mathbb Z}\frac{1}{|\beta_n|^\alpha}<\infty
\]
in the bilateral case.
Then $B_w$ admits a strongly mixing complex isotropic $\alpha$-stable invariant measure with full support.
\end{proposition}

\smallskip
\bpf
The proof follows from Theorem \ref{thm1operatorsmesurealeareelleoucomplexstable} by setting $u_n = \frac{e_n}{\beta_n}$ if $n \geq 0$ and $u_n = 0$ if $n \leq -1$ in the unilateral case, and $u_n = \frac{e_n}{\beta_n}$ for every $n \in \Z$ in the bilateral case.
\epf

\smallskip
Similarly, since $c_0(\Z_+)$ and $c_0(\Z)$ are of stable-type $\alpha$ for every $\alpha \in (0,1)$, we obtain the following one.

\smallskip

\bpr
Let $B_w$ be a unilateral or bilateral weighted shift with weight sequence $(w_n)_n$ acting on $c_0(\mathbb Z_+)$ or $c_0(\mathbb Z)$ respectively. Let $0<\alpha<1$.
Assume that
\[
\sum_{n\ge1}\frac{1}{| w_1 \dotsm w_n|^\alpha}<\infty
\]
in the unilateral case, and that
\[
\sum_{n\in\mathbb Z}\frac{1}{|\beta_n|^\alpha}<\infty
\]
in the bilateral case.
Then $B_w$ admits a strongly mixing complex isotropic $\alpha$-stable invariant measure with full support.
\epr

\smallskip

The following general result is a refinement of \cite[Proposition 5.1]{MauPrivault} to every $0 <\alpha \leq 2$. 

\smallskip

\bpr \label{propexistencemeasurestableinvariantoperator}
Let $E$ be a complex separable Banach space and let $T : E \to E$ be a bounded linear operator such that the eigenvectors of $T$ associated to unimodular eigenvalues span a dense subspace of $E$. Then, for any $0< \alpha \leq 2$, $T$ admits a complex isotropic $\alpha$-stable invariant measure with full support. 
\epr

\smallskip
\bpf
Let $(x_n)_{n \geq 1} \subset E$ be linearly independent vectors and $(\lambda_n)_{n \geq 1} \subset \T$ such that $\overline{\textrm{span}\{ x_n : n \geq 1\}} = E$ and $T x_n = \lambda_n x_n$, $\lVert x_n \rVert = \frac{1}{n}$ and $\lvert \lambda_n \rvert = 1$ for every $n \geq 1$. Let $ F : \ell^{2}(\N) \to E $ be defined by $F(e_n) = x_n$ for every $n \geq 1$. The map $F$ is a continuous linear bijection. Note that $\ell^{2}(\N)$ is of stable type $2$ and thus of stable type $\alpha$ for every $0<\alpha \leq 2 $.
Let $\alpha <2$ and $\alpha_n = \frac{1}{n^\beta}$ with $\beta > \frac{1}{\alpha}$ for every $n \geq 1$. Then $(F^{-1} T F )(\alpha_n e_n) = \lambda_n \alpha_n e_n$ for every $n \geq 1$ and $\displaystyle \sum_{n \geq 1} \lVert \alpha_n e_n \rVert_{\ell^{2}(\N)}^{\alpha} < \infty$. Moreover, $\textrm{span}\{ \alpha_n e_n : n \geq 1 \}$ is dense in $\ell^{2}(\N)$. Thus, by Theorem \ref{thm2operatormesurealeacomplexstable}, the operator $F^{-1} T F$ admits a complex isotropic $\alpha$-stable invariant measure with full support in $\ell^{\alpha}(\N)$. Let us denote by $\mu_{\alpha}$ such a measure and let $\nu_{\alpha} := \mu_{\alpha} \circ F^{-1}$. Using the characteristic function (\ref{fonccarfonccomplexmesureisotropecomplex}) of $\mu_{\alpha}$, we can easily show that $\nu_{\alpha}$ is $T$-invariant and complex isotropic $\alpha$-stable. Finally, we can easily show that the measure $\nu_{\alpha}$ has full support in $E$ using that $\mu_\alpha$ has full support in $\ell^2(\N)$. This concludes the proof of Proposition \ref{propexistencemeasurestableinvariantoperator}. 
\epf

\section{Random integrals with respect to infinitely divisible random measures}
We discuss infinitely divisible invariant measures for operators and operator semigroups. Let $\mu$ be an infinitely divisible distribution on $V = \R$ or $V = \C$, with characteristic function written for every $z \in V$ as 
\begin{align*}
    \hat{\mu}(z) = \exp(\psi(z))
\end{align*}
\enlargethispage{-6\baselineskip}
with
\begin{align*}
    \psi(z) := -\frac{1}{2} q(z) + \int_V (e^{i \Re(\overline{z}w)} -1 -i \Re(\overline{z} w) k(w)) \nu(dw),
\end{align*}
where $q : V \to [0,\infty)$ is a quadratic form, $k$ is a truncation function and $\nu$ is a measure of $V$ satisfying $\nu(\{ 0 \}) = 0$ and $\int_V \min(\lvert w \rvert^2, 1) \nu(dw) < \infty$, see \cite[pp. 7--8]{Rosinski87}. A $V$-valued infinitely divisible random measure $M$ on a $\sigma$-finite measure space $(J, \Sigma, m)$ is a random measure such that for every $A \in \Sigma$ with $m(A) < \infty$, we have 
\begin{align}
    \E(e^{i \Re(\overline{z} M(A))}) = \exp(m(A) \psi(z)), \; z \in V. 
\end{align}
We assume that $E$ is a real or complex Banach space of type $p \in [1,2]$, and that $M$ is an infinitely divisible random measure on $(J,\Sigma,m)$, taking values in $\mathbb{R}$ or $\mathbb{C}$ according as $E$ is real or complex. We also suppose that $\E(M(A)) = 0$ and that $\E(\lvert M(A) \rvert^p) \leq m(A)$ for every $A \in \Sigma$ with $m(A) < \infty$. We define a strong random integral with respect to $M$ as for stable random measures. For a simple function $f = \displaystyle \sum_{k=1}^N x_k \mathbf{1}_{A_k}$ on $J$, where $x_k \in E$ and $(A_k)_{1\le k\le N}$ are pairwise disjoint measurable sets with $m(A_k)>0$, we again set
\begin{align*}
    \int_J f(t) M(dt) = \sum_{k=1}^N x_k M(A_k).
\end{align*}
Since $E$ is of type $p$, we have by \cite[Prop. 6.7.1]{Woyczynski19}:
\begin{align*}
    \E \left \lVert \int_J f(t) M(dt) \right \rVert^p &= \E \left \lVert \sum_{k=1}^N x_k M(A_k) \right \rVert^p \\
    &\leq C \sum_{k=1}^N \lVert x_k \rVert^p \E(\lvert M(A_k) \rvert^p) \\
    &\leq C \sum_{k =1}^N \lVert x_k \rVert^p m(A_k) \\
    &\leq C \int_{J} \lVert f(t) \rVert^p m(dt)
\end{align*}
where $C>0$ is a constant depending on $p$. As for stable random measures, we can extend the definition of $\int_J f dM$ to every function $f \in L^p(J;E)$ and the inequality 
\begin{align*}
    \E \left \lVert \int_J f(t) M(dt) \right \rVert^p \leq C \int_{J} \lVert f(t) \rVert^p m(dt)
\end{align*}
again holds. Since the convergence in probability implies weak convergence, we have, for a function $f \in L^p(J;E)$ and a sequence of step functions $(f_n)_{n \geq 1}$ converging to $f$ in $L^p(J;E)$,
\begin{align*}
    \E(e^{i \Re(\lan x^*, \int_J f dM \ran)}) &= \lim_{n \to \infty} \E(\exp({i \Re \lan x^*, \sum_k x_k^n M(A_k^n) \ran))} \\
    &=\lim_{n \to \infty} \prod_k \E(e^{i \Re(\lan x^*, x_k^n \ran M(A_k^n))}) \\
    &= \lim_{n \to \infty} \exp \left ( \sum_k m(A_k^n) \psi(\overline{\lan x^*, x_k^n }\ran)\right ) \\
    &= \lim_{n \to \infty} \exp \left(\int_J \psi(\overline{ \lan x^*, f_n(t) \ran}) m(dt) \right) \\
    &= \exp \left(\int_J \psi(\overline{ \lan x^*, f(t) \ran})  m(dt) \right)
\end{align*}
by \cite[Lemma 3.1.4]{Rosinski87}. Note that if $A,B $ are two measurable disjoint sets and if $f,g \in L^p(J;E)$, then $\int_A f dM$ and $\int_B g dM$ are independent. Indeed, if $f = \displaystyle \sum_{k} x_k \indic_{A_k}$ and $g = \displaystyle \sum_{k} y_k \indic_{B_k} $ are two simple functions with $A_k \subset A,\, B_k \subset B$, $m(A_k) < \infty$ and $m(B_k) < \infty$ for each k, then the integrals $\int_A f dM$ and $\int_B f dM$ are independent. By density and since independence is preserved under weak convergence, this again holds for any $f, g \in L^p(J;E)$. Moreover, if $J$ is an interval of the real line and $m$ is the Lebesgue measure, or if $J$ is a subset of $\Z$ and $m$ is the counting measure, then for any $s \in \R$ (resp. $s \in \Z$), and for any $f \in L^p(J;E)$, the random integrals $\int_{J+s} f(t-s) M(dt)$ and $\int_J f(t) M(dt)$ have the same distribution, by the expression of the characteristic function. In this setting, we also have that $\int_0^t f(s) M(ds)$ converges in $L^p(\Omg; E)$ to $\int_{0}^\infty f dM$ as $t \to \infty$ if $J$ is an interval with the Lebesgue measure and if $f \in L^p([0,\infty); E)$, and similarly if $J$ is a subset of $\Z$ with the counting measure. Thus, the proofs of Theorems \ref{thm1espacereeloucomplexemesurealéaréelleoucomplexe} and \ref{thm1operatorsmesurealeareelleoucomplexstable} extend to this setting and gives a strongly mixing infinitely divisible invariant measure for the semigroup or the operator. Only the question on the support of the invariant measure remains open. We discuss now two examples.

\smallskip
\begin{example}
    Let $E$ be a real Banach space of type $p \in [1,2]$, let $(J,\Sigma, m)$ be a $\sigma$-finite measure space and let $\mu$ be the symmetric Poissonian distribution with Lévy measure $\nu = \frac{1}{2} (\delta_1 + \delta_{-1})$, i.e., $\hat{\mu}(t) = e^{\cos(t) -1}$ for every $t \in \R$. Let $M$ be a random measure with characteristic function
    \begin{align*}
        \E(e^{i t M(A)}) = e^{m(A)(\cos(t)-1)}, t \in \R,
    \end{align*}
    for each set $A $ of finite measure. Then $\E(M(A)) = 0$ and $\E(\lvert M(A) \rvert^p) \leq m(A)$ for every set $A$ with $m(A) < \infty$, see \cite[Page 188]{Woyczynski19}. Then, under the assumptions of Theorems \ref{thm1espacereeloucomplexemesurealéaréelleoucomplexe} and \ref{thm1operatorsmesurealeareelleoucomplexstable}, the distribution of $\int_J u_t M(dt)$ is  invariant and strongly mixing under the semigroup or the operator.
\end{example}

\smallskip

\begin{example}
    Let $E$ be a complex Banach space of type $2$, let $(J,\Sigma, m)$ be a $\sigma$-finite measure space and let $\mu$ be the symmetric Poissonian distribution with characteristic function $\hat{\mu}(z) = e^{\hat{\nu}(z) -1}$ for every $z \in \C$, where $\nu$ is a symmetric complex distribution with $\int_{\C} \lvert w \rvert^2 \nu(dw) < \infty$. Let $M$ be a random measure with characteristic function
    \begin{align*}
        \E(e^{i \Re(\overline{z} M(A))}) = e^{m(A)(\hat{\nu}(z)-1)}, z \in \C,
    \end{align*}
    for each set $A $ of finite measure.
    The measure $\mu$ is given by $\mu = e^{-1} \displaystyle \sum_{k \geq 0} \frac{\nu^k}{k!}$, called the exponent of $\nu$ (see \cite[Page 63]{Linde86}).
    Thus the distribution of $M(A)$ is the exponent of $m(A) \nu$ and we have $\E(M(A)) = 0$ and $\E(\lvert M(A) \rvert^2) = m(A) \int_{\C} \lvert w \rvert^2 \nu(dw)$ for every set $A$ with $m(A) < \infty$. For example, if $\nu$ is the uniform distribution on the unit circle of $\C$, then $\E(\lvert M(A) \rvert^2) = m(A) $ and then under the assumptions of Theorems \ref{thm1espacereeloucomplexemesurealéaréelleoucomplexe} and \ref{thm1operatorsmesurealeareelleoucomplexstable}, the distribution of $\int_J u_t M(dt)$ is invariant and strongly mixing under the semigroup or the operator. Note that in the context of Theorem \ref{thm1operatorsmesurealeareelleoucomplexstable}, the random vector $\displaystyle \sum_{n \in \Z} u_n M(\{ n \})$ has full support. Indeed, the random variables $M(\{ n \}), n \in \Z$, are i.i.d with distribution $\mu$. Moreover, the support of $\mu$ is given by $\overline{\displaystyle \bigcup_{k \geq 0} \textrm{supp}(\nu^k)}$ (see \cite[Prop. 5.8.4]{Linde86}). Since $k \cdot \T = \{ z \in \C : \lvert z \rvert \leq k \}$ for every $k \geq 2$, the distribution $\mu$ has full support, and so $\displaystyle \sum_{n \in \Z} u_n M(\{ n \})$ has full support too by \cite[Thm. 3.9]{AGNEESSENS_2023}.   
\end{example}

\bibliography{bibdossier}

\end{document}